\documentclass{amsart}
\usepackage[utf8]{inputenc}

\usepackage{xcolor}
\usepackage{amsmath}
\usepackage{amssymb}
\usepackage{comment}

\usepackage{geometry}
\allowdisplaybreaks

\title{Neural Networks in Fr\'echet spaces}
\author{Fred Espen Benth, Nils Detering, Luca Galimberti}
\address{Fred Espen Benth \\
University of Oslo\\
Department of Mathematics \\
P.O. Box 1053, Blindern\\
N--0316 Oslo, Norway}
\email[]{fredb\@@math.uio.no}
\address{Nils Detering \\ 
University of California at Santa Barbara\\
Department of Statistics and Applied Probability\\
CA 93106 Santa Barbara, USA}
\email[]{detering\@@pstat.ucsb.edu}
\address{Luca Galimberti \\ 
Norwegian University of Science and Technology\\
Department of Mathematical Sciences\\
Sentralbygg 2, Gl\o shaugen, Trondheim, Norway}
\email[]{luca.galimberti\@@ntnu.no}

\newtheorem{theorem}{Theorem}[section]
\newtheorem{definition}[theorem]{Definition}
\newtheorem{lemma}[theorem]{Lemma}
\newtheorem{proposition}[theorem]{Proposition}

\newtheorem{corollary}[theorem]{Corollary}

\newtheorem{remark}[theorem]{Remark}
\newtheorem{example}[theorem]{Example}

\newcommand{\R}{\mathbb R}
\newcommand{\C}{\mathbb C}
\newcommand{\N}{\mathbb N}

\newcommand{\T}{\mathcal T}
\renewcommand{\P}{\mathbb P} %Probability measure
\newcommand{\EE}{\mathbb E} %Expected value

\newcommand{\norm}[1]{\left\lVert#1\right\rVert}
\newcommand{\abs}[1]{\left |#1\right|}
\newcommand{\X}{\mathfrak X}
\newcommand{\Y}{\mathfrak Y}

\DeclareMathOperator{\Span}{span}
\DeclareMathOperator{\Supp}{supp}
\newcommand{\ft}{F^t}
\newcommand{\fth}{F^{t,H,\phi}}

\newcommand{\B}{\mathcal B}

\newcommand{\nils}[1]{\textcolor{blue}{#1}}
\newcommand{\luca}[1]{\textcolor{orange}{#1}}
\newcommand{\fred}[1]{\textcolor{red}{#1}}

\date{\today}

\begin{document}
\begin{abstract}
We define a neural network in infinite dimensional spaces for which we can show the universal approximation property. Indeed, we derive approximation results for continuous functions from a Fr\'echet space $\X$ into a Banach space $\Y$. The approximation results are generalising the well known universal approximation theorem for continuous functions from $\mathbb{R}^n$ to $\mathbb{R}$, where approximation is done with (multilayer) neural networks \cite{Cybenko1989,HORNIK1989359,FUNAHASHI1989183,LESHNO1993861}. Our infinite dimensional networks are constructed using activation functions being nonlinear operators and affine transforms. Several examples are given of such activation functions. We show furthermore that our neural networks on infinite dimensional spaces can be projected down to finite dimensional subspaces with any desirable accuracy, thus obtaining approximating networks that are easy to implement and allow for fast computation and fitting. The resulting neural network architecture is therefore applicable for prediction tasks based on functional data.
\end{abstract}
\maketitle

\section{Introduction}
The universal approximation theorem shows that any continuous function from $\mathbb{R}^n$ to $\mathbb{R}$ can be approximated arbitrary well with a one layer neural network. More precisely, for a fixed continuous function $\sigma : \mathbb{R}\rightarrow \mathbb{R}$ and $a \in \mathbb{R}^n, \ell , b \in \mathbb{R}$, a {\em neuron} is a function $\mathcal{N}_{\ell , a,b}\in C(\mathbb{R}^n ; \mathbb{R})$ defined by $x\mapsto \ell \sigma (a^{\top}x +b) $. The universal approximation theorem states conditions on the {\em activation function} $\sigma$ such that the linear space of functions generated by the neurons 
$$\mathfrak N(\sigma):=\Span \{ \mathcal{N}_{\ell , a,b}; \ell ,b \in \mathbb{R} , a\in \mathbb{R}^n \} $$ is dense with respect to the topology of uniform convergence on compacts. This means that for every $f\in C(\mathbb{R}^n; \mathbb{R})$ and compact subset $K\subset \mathbb{R}^n$ and a given $\varepsilon >0$, there exists $N\in \mathbb{N}$ and $\ell_i ,b_i \in \mathbb{R} , a_i\in \mathbb{R}^n$ for $i=1,\dots , N$ such that
$$ \sup_{x\in K}\abs{f(x) - \sum_{i=1}^{N} \mathcal{N}_{\ell_i , a_i,b_i}  (x)}\leq \varepsilon .$$
Possibly the most widely known property of $\sigma$ that was shown in Cybenko \cite{Cybenko1989} and Hornik, Stinchcombe, and White \cite{HORNIK1989359} to lead to the density of $\mathfrak N(\sigma)\subset C(\mathbb{R}^n ; \mathbb{R})$ is the {\em sigmoid} property, which requires $\sigma$ to be such that $\lim_{t\rightarrow \infty} \sigma(t)=1$ and $\lim_{t\rightarrow -\infty} \sigma(t)=0$. This condition has later been relaxed to a boundedness condition Funahashi \cite{FUNAHASHI1989183} and a non-polynomial condition Leshno {\it et al.} \cite{LESHNO1993861}. We refer the reader to Pinkus \cite{pinkus_1999} for an overview of the earlier literature on neural network approximation theory and to Berner {\it et al.} \cite{MathFoundationsDeep} for a more recent account. See also Kratsios \cite{UniversalApproxProperty} for a unified approach of approximation result for a wide class of network architectures. 

In this paper we are concerned with more general functions $f\in C(\X;\Y)$, where $\X$ is an $\mathbb{F}$-Fr\'echet space, i.e., a Fr\'echet space over the field $\mathbb F$ and $\Y$ an $\mathbb{F}$-Banach space.
We start with $\Y=\mathbb{F}$ and in the definition of a neuron, we replace $a^{\top}x +b$ by an affine function on $\X$, the activation function $\sigma : \mathbb{R} \rightarrow \mathbb{R}$ by a function in $C(\X;\X)$, and the scalar $\ell$ by a linear form. With $\langle\cdot,\cdot\rangle$ the canonical pairing between $\X'$ and $\X$ ($\X'$ denoting the topological dual of $\X$), for $\ell\in \X', A\in \mathcal{L} (\X), b\in \X$ we then define a neuron $\mathcal{N}_{\ell,A,b}$ by 
$$\mathcal{N}_{\ell,A,b}(x)= \langle \ell ,\sigma (Ax +b)\rangle$$
and ask for conditions on $\sigma:\X \rightarrow \X$ that ensure that $\mathfrak N(\sigma):=\Span \{ \mathcal{N}_{\ell , A,b}; \ell\in \X' ,A\in \mathcal{L} (\X) ,b \in \X  \} $ is dense in $C(\X;\mathbb{F})$ under some suitable topology. 

To indicate the conditions we obtain for $\sigma$, recall that any map $\psi \in \X'$ defines a hyperplane by the set of points $\Psi_0:= \{x\in \X; \langle \psi ,x \rangle =0 \}$. This hyperplane splits the space $\X$ into the sets $\Psi_{-}:= \{ x\in \X; \langle \psi ,x \rangle < 0 \}$ and $\Psi_{+}:=\{ x\in \X; \langle \psi ,x \rangle >0\}$. We show that the main property for the activation function to ensure that $\mathfrak N(\sigma)$ is dense in $C(\X; \mathbb{F})$ is, informally, that an $\psi \in \X'$ exists such that the value $\sigma (x)$ converges, as $x$ moves away from the hyperplane that is defined by $\psi$. The limiting values on both sides of the hyperplane need to be different. We provide several simple examples of easy to calculate activation functions with the required property.  In a second step, we extend our results to $f\in C(\X;\Y)$, where $\Y$ is an $\mathbb{F}$-Banach space. 

While such an approximation result might be of interest in its own, from a practical perspective it is not clear how the functions $\mathcal{N}_{\ell,A,b}$, which involve infinite dimensional quantities, can actually be programmed. We therefore address the question of approximating the maps $\mathcal{N}_{\ell,A,b}$ by finite dimensional, easy to calculate quantities. Under the assumption that the Fr\'echet space $\X$ admits a Schauder basis, we show that such an approximation is possible. The resulting neural network has an architecture similar to classical neural networks, with the exception that the activation function is now multidimensional. It does however still permit for an easy to calculate gradient, which is crucial for training the network via a back-propagation algorithm. Finally, we also derive the approximation property for deep neural networks with a given fixed number of layers.

 We emphasise that our proposed definition of a neural network in infinite dimensions is motivated by the relationship with controlled ordinary differential equations, which points towards an activation function $\sigma : \X \rightarrow \X$ rather than the classical one-dimensional maps (possibly on basis coordinates). We refer to E \cite{WeinanE} for a connection between ordinary differential equations and deep neural networks, as well as Section \ref{sec:multilayer} in this paper.

Possible applications of our results are within the area of machine learning, in particular in the many situations where the input of each sample in the training set is actually a function (see e.g. Ramsey and Silverman \cite{RamSilver} for an account on functional data analysis and examples). In our accompanying paper \cite{BDG3} we use the results obtained here to derive numerical solutions of partial differential equations for a range of initial conditions or coefficients at once (see Han, Jentzen and E \cite{Han8505}, Hutzenthaler {\it et al.} \cite{dneuralnetPDE}, Cuchiero, Larsson and Teichmann \cite{Cuchiero2020DeepNN}, Beck {\it et al.} \cite{beck2021overview} for papers on neural networks and partial differential equations). 
There are other instances where functional data appears naturally. For example grey scale images can be understood as a function $I: [0,1]^2 \rightarrow [0,1]$. For imagine classification or recognition problems (see M\"uller, Soto-Ray and Kramer \cite{https://doi.org/10.48550/arxiv.2201.11440} and Tian \cite{brainImage}) one is now interested in approximating the function $f$ that assigns to each image its classification $f(I)$. Additional examples are stock price prediction (see Yu and Yan \cite{stock:predition}), option pricing and hedging (see Buehler {\it et al.} \cite{doi:10.1080/14697688.2019.1571683} and Benth, Detering and Lavagnini \cite{accuracyBDL}), and many others. 

If the function space of the inputs is a Fr\'echet space with a Schauder basis, this basis provides structural information about the elements. Traditional neural networks must be of very high dimension (large input dimension, large number of neurons) to approximate a function well. The more variability there is in the function, the larger the number of parameters that is needed. Therefore, instead of using a classical network to approximate a function on a grid, our approach allows one to use information in the basis functions instead to capture the structure and get theoretical convergence results. Our approximation thus focuses on features of the function related to the coefficients in the basis expansion. Moreover, we show that there is a large class of possible activation functions $\sigma : \X \rightarrow \X$ and a choice that is suitable for the approximation problem at hand can significantly reduce the number of nodes required to approximate a given function $f$ sufficiently well. We refer to our accompanying paper \cite{BDG2} where this idea is used to price flow forward derivatives in energy markets.

{\bf Related literature:} The approximation with neural networks of functionals and operators that are defined on some general (possibly infinite dimensional) space $\X$ goes back to Sandberg \cite{76498}. In Sandberg \cite{76498} in the context of discrete time systems, non-linear functionals on a space of functions from $\mathbb{N} \cup \{ 0\}$ to $\mathbb{N}$ are approximated with neural networks. In Chen and Chen \cite{286886, 392253} the authors consider the approximation of non-linear operators defined on infinite dimensional spaces and use these results for approximating the output of dynamical systems. Among other results they approximate functions $f: K\subset \X \to \R$, where $\X$ is Banach, $K$ is compact and $f$ is continuous.
%; ii) or $T:V \to C(K_2)$, where $V\subset C(K_1)$ is compact, $K_1\subset \X$ is compact, $\X$ is Banach, $K_2\subset\R^n$ is compact, and $T$ is continuous}
In Mhaskar and Hahm \cite{6795742} the authors derive networks that approximate the functionals on the function spaces $L^p ([-1,1]^s)$ for $1 \leq p < \infty $ and $C ([-1,1]^s)$ for integer $s\geq 1$. The network architectures in all these works differ slightly but they have in common an activation function $\sigma$ with image in $\mathbb{R}$ instead of $\X$ as we propose it here. The recent article by Kratisos \cite{UniversalApproxProperty} considers a space $M(\X,\Y)$ of functions from a metric space $\X$ to another metric space $\Y$. Among other results, under the assumption that this functions space is homeomorphic to an infinite-dimensional Fr\'echet space, the author derives properties of neural network architectures that are dense within this space. We would like to stress however that in the situation we consider in this paper, the domain space $\X$ is a Fr\'echet space. The function space $C(\X, \mathbb{R})$ however is usually not a Fr\'echet space unless $\X$ is finite dimensional. 
Infinitely wide neural networks, with an infinite but countable number of nodes in the hidden layer have been studied in the context of Bayesian learning, Gaussian processes and kernel methods by several authors, see e.g., Neal \cite{Neal}, Williams \cite{Williams}, Cho and Saul \cite{ChoSaul} and Hazan and Jaakola \cite{HJ}. Hornik \cite{Hornik-93} provides approximation results for such infinitely wide networks. Guss and Salakhutdinov \cite{Guss} prove the universal approximation property for two-layer infinite dimensional neural networks. They show this for continuous maps between spaces of continuous functions on compacts. We also refer the reader to Kratsios and Bilokopytov \cite{inproceedings} for approximations on manifolds in $\mathbb{R}^n$.

The outline for the paper is as follows. In Section~\ref{main:approx} we derive our first main result Theorem~\ref{prop: density}, which shows that if $\sigma$ has a property called {\em discriminatory} property, then $\mathfrak N(\sigma)$ is dense in $C(\X;\mathbb{F})$. The main technical challenge is then to derive conditions that ensure that a given function $\sigma : \X \rightarrow \X$ is actually discriminatory, which is done in Theorem~\ref{prop: discriminatory condition is satisfied}. We also provide some first examples of discriminatory functions in this section. We then extend these results in Section~\ref{sec:BanachvaluedNN} to functions $f\in C(\X;\Y)$, $\Y$ Banach space. In Section~\ref{finte:approx:section} we address the question of finite dimensional approximations to the neural network which can easily be computed and trained. In most generality, only under the assumption that the Fr\'echet space $\X$ has a Schauder basis, the approximation is covered in Theorem~\ref{prop: finite dimensional approx, Frechet}. In Section~\ref{sec:multilayer} we cover the approximation with multi-layered neural networks. 

\section*{Acknowledgements}
Fred Espen Benth acknowledges support from SPATUS, a Thematic Research Group funded by UiO:Energy.

Luca Galimberti has been supported in part by the grant Waves and Nonlinear Phenomena (WaNP) from the Research Council of Norway.

\section{An abstract approximation result}\label{main:approx}

Let $\mathbb{F}\in \{\R,\C\}$, and let $\mathfrak X$ be an $\mathbb F$-Fr\'echet space. Let $(p_k)_{k\in\N}$ be an increasing sequence of seminorms that generates the topology of $\mathfrak X$. We can then consider a metric $d$ on $\mathfrak X$ (that generates the same topology) given by
\begin{equation}\label{metric:loc:conv:space}
    d(x,y):= \sum_{k=1}^\infty 2^{-k}\frac{p_k(x-y)}{1+p_k(x-y)}, %\quad x,y\in \mathfrak X.
\end{equation}
for $x,y\in \mathfrak X$.
%and note that $\abs{x}:=d(x,0)$ defines a pseudo-norm on $\X$. \luca{I do not think we need this remark, since we do not use it at all in the sequel}\nils{Yes, you mean the part with the preudo-norm?}

Let us consider $\sigma: \X \to \X$ continuous function. Let $A:\X\to \X$ be in $\mathcal{L}(\X)$, i.e. a linear and continuous operator, $b\in\X$ and $\ell\in \X'$, where $\X'$ denotes the topological dual of $\X$. Let us consider the following function:
\begin{equation}\label{eq: definition of 1 layer NN}
    \mathcal{N}_{\ell,A,b} : \X\to \mathbb F, \quad \mathcal{N}_{\ell,A,b}(x):= \langle \ell,\sigma(Ax+b)\rangle = \ell(\sigma(Ax+b)), \quad x\in\X,
\end{equation}
where $\langle\cdot,\cdot\rangle$ is the canonical pairing between $\X'$ and $\X$. We will call such function a {\it neuron}. %\nils{Should that not be called a node rather than a layer? The layer would then be the collection of nodes $j=1, \dots, N$ as appearing in the sum below. I think this is a bit more in line with \cite{Cybenko1989} and others where the affine functions build the nodes.}
Every neuron $\mathcal{N}_{\ell,A,b}$ is clearly continuous by composition of continuous maps, i.e. $\mathcal{N}_{\ell,A,b}\in C(\X;\mathbb F)$, the space of $\mathbb F$-valued continuous functions on $\X$.

We define
\begin{equation*}
    \mathfrak N(\sigma) := \Span\{
    \mathcal{N}_{\ell,A,b}; \,\,\ell\in\X', A\in\mathcal{L}(\X),b\in\X    \},
\end{equation*}
namely, we consider all linear combinations of the form
\begin{equation*}
    \sum_{j=1}^N \alpha_j \mathcal{N}_{\ell_j,A_j,b_j}, \quad \alpha_j\in\mathbb F,N\in\N.
\end{equation*}
Evidently, $\mathfrak N(\sigma)\subset C(\X;\mathbb F)$. %, the space of $\mathbb F$-valued continuous functions on $\mathfrak X$. 
The maps $\mathcal{N}_{\ell_1,A_1,b_1}, \dots , \mathcal{N}_{\ell_N,A_N,b_N}$ build a {\it hidden layer} with $N$ neurons. 

We endow $C(\X;\mathbb F)$ with the topology of uniform convergence on compacts. Being $\X$ metrizable, it is clearly Tychonoff, and in particular completely regular. %We will apply Riesz representation theorem directly to $C(\X;\mathbb F)$. %(Note that now we do not need anymore to assume any boundedness assumption on the non-linearity $\sigma$ at this stage). 
%In order to accomplish that, we need to use a ``non-standard'' topology on $C(\X;\mathbb F)$. Let's see how it works: 
For a given compact subset $K\subset \X$, define
\begin{equation*}
    q_K(f) := \sup_{x\in K}\abs{f(x)},\quad f \in C(\X;\mathbb F).
\end{equation*}
This is a seminorm on $C(\X;\mathbb F)$. We consider the topology generated by the family of seminorms $\{q_K; K\subset \X, \text{ compact} \}$, which is the coarsest topology that makes all the seminorms continuous functions on $C(\X;\mathbb F)$. This is also called the projective topology induced by the maps $q_K$ for $K$ compact or the topology of compact subsets. Thus, we obtain a locally convex topology on $C(\X;\mathbb F)$, namely $C(\X;\mathbb F)$ is an $\mathbb F$-locally convex space. Conway \cite[Proposition 4.1, p. 114]{conway2010} provides us with the following Riesz representation theorem, which we are going to employ in the sequel:

\begin{proposition}\label{Riesz}
If $\phi:C(\X;\mathbb F)\to \mathbb F$ is a continuous and linear functional, then there is a compact set $K\subset\X$ and a regular Borel measure $\mu$ on $K$ such that $\phi(f)=\int_Kf\,d\mu$ for every $f\in C(\X;\mathbb F)$. Conversely, each such measure defines an element of $C(\X;\mathbb F)'$. (Observe en passant that $\abs{\mu}(K)<\infty$.)
%\fred{Does this measure also require a compact $K$? I.e., it is each {\bf pair} $(\mu,K)$ that defines a functional? }\luca{The measure can be in principle defined on the whole space. But to ensure that we end up with a finite integral, we need to select in any case a compact subset.}\nils{Is it not that simply each measure on $\X$ that has compact support defines an element in $C(\X;\mathbb F)'$? I guess if one thinks about the pairs $(\mu,K)$ then the issues is that $\mu$ might have support actually on a smaller compact set $K'\subset K$ and the pair $(\mu ',K')$ where $\mu '$ is $\mu$ restricted to $K'$ defines the same functional.}
\end{proposition}

We recall that for a locally compact space $Y$ equipped with its Borel $\sigma$-algebra $\mathcal B(Y)$, a positive measure $\nu$ on $\mathcal{B}(Y)$ is a regular Borel measure if
\begin{enumerate}
    \item $\nu(F)<\infty$ for every $F\subset Y$ compact,
    \item for any $E\in\mathcal{B}(Y)$, $\nu(E)=\sup\{\nu(F); F\subset E, F \text{ compact}\}$,
    \item for any $E\in\mathcal{B}(Y)$, $\nu(E)=\inf\{\nu(U); U\supset E, U \text{ open}\}$.
\end{enumerate}
If $\nu$ is complex-valued or signed instead, then it is regular if $\abs{\nu}$ is.

In the following the expression $(\mu,K)$ will denote a compact subset $K\subset\X$ and a regular $\mathbb F$-valued Borel measure $\mu$ on $K$. We  say that $\sigma:\X\to\X$ continuous is discriminatory if for any fixed pair $(\mu,K)$ 
\begin{equation*}
    \int_K \langle \ell,\sigma(Ax+b)\rangle\, \mu(dx) = 0
\end{equation*}
for all $\ell\in\X',A\in\mathcal L(\X),b\in\X$ implies that $\mu=0$. 
\begin{remark}
It would be tempting, albeit more challenging, to establish our universal approximation result (Thm \ref{prop: density}) in a ``global'' setting, namely to work directly in the space of bounded continuous functions $C_b(\X;\mathbb F)$, endowed with the supremum norm (upon imposing suitable boundedness conditions on the non-linearity $\sigma$), rather than staying at a ``local'' level as we are doing now. 

The main obstruction that prevented us from employing this approach is explained by the succeeding observation: If we aim at following Cybenko's blueprint \cite{Cybenko1989}
(refer to the proof of Thm. \ref{prop: density} below) to establish our result, then in that case we would be required to work with the space 
\begin{equation*}
    rba(\X):= \{
    \mu:\mathcal{B}(\X)\to \mathbb F; \; \mu(\emptyset) =0, \text{finitely additive, finite and regular}
    \}
\end{equation*}
which is known to be the dual of $C_b(\X;\mathbb F)$, i.e. $C_b(\X;\mathbb F)'=rba(\X)$. Dealing with finitely additive measures is more involved, because many standard results from classical measure theory cease to hold. In particular, at this stage it is not clear to us to envisage a suitable set of conditions that the non-linearity $\sigma$ must satisfy in order to be discriminatory (see Def. \ref{sigmoid}).  

Nonetheless, we deem this potential extension of our result to be interesting and worthy to be explored (most likely by deviating completely from Cybenko's strategy of proof), and we hope to be able to come back to this question in the future.  
\end{remark}

The following first main result shows the density of $\mathfrak N(\sigma)$ if $\sigma$ is discriminatory. The result takes inspiration from Cybenko \cite{Cybenko1989} (see also \cite{FUNAHASHI1989183}, \cite{HORNIK1989359} and \cite{LESHNO1993861}), where a similar result has been shown for the case $\X = \mathbb{R}^n$. For general $\X$ however, showing that a function $\sigma:\X\to\X$ is actually discriminatory can be involved. Later, in Theorem~\ref{prop: discriminatory condition is satisfied} we therefore state conditions that can easily be verified and give rise to a large family of discriminatory functions. 
\begin{theorem}\label{prop: density}
Let $\X$ be an $\mathbb F$-Fr\'echet space, and let $\sigma:\X\to\X$ be continuous and discriminatory. Then $ \mathfrak N(\sigma)$ is dense in $C(\X;\mathbb F)$ when equipped with the projective topology with respect to the seminorms $q_K$. In other words, given $f\in C(\X;\mathbb F)$, then, for any compact subset $K$ of $\X$, and any $\varepsilon>0$, there exists $ \sum_{m=1}^M \alpha_m\mathcal{N}_{\ell_m,A_m,b_m}\in  \mathfrak N(\sigma)$ with suitable $\alpha_m\in\mathbb F, \ell_m\in\X',A_m\in\mathcal{L}(\X)$ and $b_m\in\X$ such that
\begin{equation*}\label{approx:prop}
    \sum_{m=1}^M \alpha_m\mathcal{N}_{\ell_m,A_m,b_m} \in 
\{ g \in C(\X;\mathbb F);\; q_{K}(g-f)<\varepsilon\}.
\end{equation*}
\end{theorem}
\begin{proof}

We assume that $\text{cl}(\mathfrak N(\sigma)) \subsetneq C(\X;\mathbb F)$, and observe that $\text{cl}(\mathfrak N(\sigma))$ is clearly still a vector subspace.

We choose $u_0\in C(\X;\mathbb F)\setminus \text{cl}(\mathfrak N(\sigma))$. Since the complement of $\text{cl}(\mathfrak N(\sigma))$ is open, we may find $n\in \mathbb{N}$, seminorms $q_{K_1},\dots ,q_{K_n}$ on $C(\X;\mathbb F)$ and $\varepsilon_1,\dots ,\varepsilon_n>0$ such that
\begin{equation*}
\mathcal{U}:= \bigcap_{j=1}^n 
\{ u \in C(\X;\mathbb F);\; q_{K_j}(u-u_0)<\varepsilon_j\} \subset C(\X;\mathbb F) \setminus \text{cl}(\mathfrak N(\sigma)) .
\end{equation*}

Clearly $u_0\in\mathcal U$, $\mathcal U$ is convex, open and disjoint from $\text{cl}(\mathfrak N(\sigma))$. From one of the Corollaries of the Hahn-Banach Theorem (see e.g. Narici \cite[Thm. 8.5.4]{alma990006098300203776}) there exists $\phi: C(\X;\mathbb F)\to \mathbb F$ linear and continuous such that
\begin{equation*}
    \phi\big|_{\operatorname{cl}(\mathfrak N(\sigma))} = 0, \quad \Re(\phi) >0\,\, \text{ on }\, \mathcal{U}.
\end{equation*}
In particular, $\phi$ is not identically zero. Then by Proposition~\ref{Riesz}, there exists a compact subset $K\subset\X$ and a regular Borel measure (complex or signed) $\mu\neq 0$ on $K$ %($\mu\neq 0$) 
such that
\begin{equation*}
    \phi(f)= \int_K f(x)\,\mu(dx), \quad f \in C(\X;\mathbb F).
\end{equation*}
In particular, for any $\ell\in\X',A\in\mathcal L(\X),b\in\X$ it holds
\begin{equation*}
    \int_K \langle \ell,\sigma(Ax+b)\rangle\, \mu(dx) = 0.
\end{equation*}
But $\sigma$ was assumed to be discriminatory. Thus we infer $\mu=0$, and this is a contradiction to $\Re(\phi) >0$ on $\mathcal{U}$. We conclude that $\mathfrak N(\sigma)$ is dense in $C(\X;\mathbb F)$ with respect to the topology of compact subsets of $\X$. This implies that there exits $M$ and $\alpha_m\in\mathbb F, \ell_m\in\X',A_m\in\mathcal{L}(\X)$ and $b_m\in\X$ for $m=1,\dots, M$ such that (\ref{approx:prop}) holds.
\end{proof}

%\begin{remark}

%Since $(\X,d)$ is a metric space, we know that any compact set $K\subset\X$ is closed and totally bounded, namely for each $\varepsilon>0$ it can be covered by a finite number of open balls of radius $\varepsilon$. To give an example of a truly infinite dimensional compact subset, let us . \nils{Shall we add here Lucas Example of a truly infinite dimensional compact set?}
%\end{remark}
\begin{example}
Because Theorem~\ref{prop: density} allows us to approximate continuous functions on compact subsets of $\X$ with neural networks, let us outline a typical example of an infinite dimensional compact subset. First recall that for $\X$ Banach space, a subset $S\subset \X$ is compact if and only if 
%\nils{Do we need reference here?}\luca{Yes... but at the end of the day, this is just a corollary of the fact the compact subsets are closed and totally bounded. Look e.g. at this: https://faculty.etsu.edu/gardnerr/Func/notes/9-2.pdf}\nils{Yes, I see. Actually I think we do not need a reference then.}
(i) $S$ is closed and bounded, (ii) for all $\varepsilon >0$, there exists a finite dimensional subspace $\X_{\varepsilon} \subset \X$ such that for all $s\in S$, it holds that $d(s,\X_{\varepsilon})< \varepsilon$. Let now $\X$ be a separable Hilbert space and let $( e_k)_{k\in \N}$ be an orthonormal 
basis for $\X$. Then every $x\in \X$ can be represented as $x=\sum_{k=1}^{\infty}x_k e_k$ with coefficients $x_k\in \mathbb{F}$. Let us choose $(s_k)_{k\in \N}\in \ell^2$ with $s_k\geq 0$ for all $k\in \N$. Here $\ell^2$ denotes the space of square integrable sequences. The set 
\begin{equation}
    S:=\{x\in \X : \abs{x_k} \leq s_k,\;\; \forall k\in \N \} 
\end{equation}
is then compact. To see this, first observe that $S$ is clearly bounded. Now, let $y\in cl(S)$. Then we may find a sequence $(x(n))_{n\in \N}$ in $S$ such that $x(n)$ converges to $y$. This in particular means that $x_k(n)$ converges to $y_k$ for all $k\in \N$. But this implies that $\abs{y_k}\leq s_k$ and hence $y\in S$ and $S$ is closed (i.e., (i) holds). Finally, let $\varepsilon > 0$, then choose $N_{\varepsilon}\in\mathbb N$ such that $$\sum_{k=N_{\varepsilon}+1}^{\infty} s_k^2 < \varepsilon^2$$ and set $\X_{\varepsilon}:=\text{span} \{e_1, \dots , e_{N_{\varepsilon}} \}$, which is clearly finite dimensional. For any $x\in S$ it holds that $$\norm{x-\sum_{k=1}^{N_{\varepsilon}} x_k e_k}^2=\sum_{k=N_{\varepsilon}+1}^{\infty}x_k^2 < \varepsilon^2 ,$$ which clearly implies that $d(x, \X_{\varepsilon}) \leq \norm{x- \sum_{k=1}^{N_{\varepsilon}}x_k e_k}< \varepsilon$ and hence (ii) holds.
\end{example}

For the sequel, we need a boundedness assumption on the activation function $\sigma$. First, recall that a set $A\subset\mathfrak{X}$ is von Neumann-bounded if for any $k\in \N$ there exists $c_k>0$ such that $\sup_{x\in A}p_k(x)\leq c_k$. 
We assume that the set 
\begin{equation}
    \sigma(\X)\subset \X
\end{equation}
is von Neumann-bounded. 
\begin{remark}
We have another concept of metric-boundedness available:
A subset $A$ of a metric space $(\X, d)$ is bounded if there exists $R>0$ such that for all $x_1,x_2 \in A$ it holds $d(x_1,x_2)<R$.
This concept is not sufficiently stringent, because $diam(\X)\leq 1$ under the metric defined in \eqref{metric:loc:conv:space}, and thus any subset of $\X$ is bounded. von Neumann-boundedness is more well-suited when one works with metrizable topological vector spaces. 
\end{remark}
%We therefore assume that
%\begin{equation*}
 %   \sigma(\X)\subset \X
%\end{equation*}
%is (von Neumann)-bounded.  
Assuming von Neumann-boundedness is convenient because it enables us to interchange limits and integrals. Observe that in the case in which $\X$ is normed, we are back to the classical concept of boundedness. 

%\fred{Should we make an explicit assumption (or definition) where we state what we mean by boundedness? Following the definition, we can make a remark with the alternative (metric-boundedness) and why this is not appropriate for our purposes?}

In view of the von Neumann-boundedness assumption on $\sigma$, for any $\ell\in\X',A\in\mathcal L(\X),b\in\X$
\begin{equation*}
    \abs{\mathcal{N}_{\ell,A,b}(x)}\leq C_\ell \,p_{k_\ell} (\sigma(Ax+b)),\quad x\in \X
\end{equation*}
for some constant $C_\ell\geq 0$ (compare Schaefer \cite[Thm. 1.1, p. 74]{schaefer1971topological}), and thus, for a constant $C(\ell,\sigma)$ depending on $\ell$ and $\sigma$
\begin{equation*}
    \abs{\mathcal{N}_{\ell,A,b}(x)}\leq C(\ell,\sigma),\quad x\in \X.
\end{equation*}

We next investigate under which conditions a non-linear function $\sigma$ is discriminatory. From now on, we assume that $\mathbb F=\R$, because we need that hyperplanes disconnect the space $\X$. If $\mathbb F=\C$, this of course, cannot hold.

We now state a condition that ensures that $\sigma$ is discriminatory. In order to develop some intuition for this condition, first recall that any $\psi\in \X'\setminus\{0\}$ defines a hyperplane in $\X$ by the set $\Psi_0=\ker(\psi)$. This hyperplane splits $\X$ between the two half-spaces $\Psi_+ =\{ x\in\X; \langle\psi,x\rangle >0 \}$ and $\Psi_- =\{ x\in\X; \langle\psi,x\rangle <0 \}$, which lie on either side of the hyperplane. It turns out that measures on $\mathcal{B} (\X)\cap K$ are fully determined by their values on the half-spaces arising from all shifted hyperplanes. If now $\sigma$ splits the space $\X$ in the sense that there exists one particular hyper-plane $\Psi_0$ such that on either side of this hyperplane, the function $\sigma(\lambda x)$ converges as $\lambda \rightarrow \infty$, then  this implies that $\sigma(\lambda x)$ converges pointwise to a function that is constant on both half-spaces separated by $\Psi_0$. Integrating this pointwise limit over either of those spaces determines the value of the measure on them. The maps  $A\in\mathcal L(\X),b\in\X$ now allow to rotate, shift and project to all possible half-spaces and determine the measure on them (see Lemma~\ref{lemma: replication of linear functionals}).

The following separating property is the infinite-dimensional counterpart to the well known sigmoidal property for functions from $\mathbb{R}$ to $\mathbb{R}$ (see Cybenko \cite{Cybenko1989}):
\begin{definition}{Separating property:}\label{sigmoid}
%\texttt{Condition on $\sigma$:}
There exist $\psi\in \X'\setminus\{0\}$ and $u_+,u_-,u_0\in\X$ such that either $u_+ \notin \Span \{u_0,u_-\}$ or $u_- \notin \Span \{u_0,u_+ \}$ and such that 
\begin{equation}\label{eq: abstract condition on sigma}
\begin{cases}
\lim_{\lambda\to\infty} \sigma(\lambda x) = u_+, \text{ if } x\in \Psi_+\\
\lim_{\lambda\to\infty} \sigma(\lambda x) = u_-, \text{ if } x\in \Psi_-\\
\lim_{\lambda\to\infty} \sigma(\lambda x) = u_0, \text{ if } x\in \Psi_0\\
\end{cases}
\end{equation}
where we have set as above
\begin{equation*}
    \Psi_+ =\{ x\in\X; \langle\psi,x\rangle >0 \}, \quad \Psi_- =\{ x\in\X; \langle\psi,x\rangle <0 \}
\end{equation*}
and $\Psi_0=\ker(\psi)$.
\end{definition}
We point out that as a particular case of the Separating property we may choose $u_0=u_-=0$ and $u_+\neq 0$ for instance. We now provide a first example of a function $\sigma$ that fulfills the separating property. It is in the spirit of the classical Sigmoid activation function. More examples are provided in Section~\ref{additional:examples}.
\begin{example}\label{ex: abstract condition on sigma}
We are going to give a construction of a continuous and von Neumann-bounded function $\sigma:\X\rightarrow\X$ satisfying the Separating property in Definition~\ref{sigmoid}, for $u_+,u_-,u_0\in \X$ such that either $u_+ \notin \Span \{u_0,u_-\}$ or $u_- \notin \Span \{u_0,u_+ \}$.

Let us recall this abstract result first: given a metric space $(Z,d)$ and $\emptyset \neq Y\subset Z$, define
\begin{equation*}
    F_\varepsilon(x) :=\max (1-\varepsilon^{-1}d(x,Y),0),\quad x\in Z,\,\varepsilon>0.
\end{equation*}
Then $F_\varepsilon$ is Lipschitz continuous, $F_\varepsilon\in[0,1]$ and $F_\varepsilon(x)\to I_Y(x)$ for any $x\in Z$ as $\varepsilon\to 0$.

Consider $\psi\in \X'\setminus\{0\}$ arbitrary. We approximate with this trick the indicator functions $I_{\{\psi \geq 1\}}$, $I_{\{\psi \leq -1\}}$ and $I_{\{\psi =0\}}$, obtaining respectively $F_{\varepsilon,1}$, $F_{\varepsilon,-1}$ and $F_{\varepsilon,0}$. The scaling parameter $\varepsilon$ is chosen small enough such that the supports of these functions do not meet. This is clearly possible. Indeed: suppose first that $d(\{\psi=1\},\{\psi =0\})=0$. Then we might find $(z_n,y_n)\in\{\psi=1\}\times\{\psi=0\}$ such that $d(z_n,y_n)\to 0$, namely $p_k(z_n-y_n)\to 0$ for any $k\in\N$. But on the other hand, for some $j\in\N$ and $c_j>0$
\begin{equation*}
   1= \abs{\langle\psi,z_n\rangle - \langle\psi,y_n\rangle} \leq c_j p_j(z_n-y_n)\to 0 
\end{equation*}
and thus $d(\{\psi=1\},\{\psi =0\})>0$. Since $\Supp F_{\varepsilon,1} = \operatorname{cl}(\{\psi\geq 1\}_{\varepsilon})$ and $\Supp F_{\varepsilon,0} = \operatorname{cl}(\{\psi=0\}_{\varepsilon})$ (for an arbitrary subset $Y$, $Y_\varepsilon$ denotes its $\varepsilon$-neighborhood), for $4\varepsilon < d(\{\psi=1\},\{\psi =0\})$ we obtain that the supports do not meet. The same holds for the other cases.

Define
\begin{equation*}
    \sigma(x):= F_{\varepsilon,1}(x)u_+ + F_{\varepsilon,-1}(x)u_-  + F_{\varepsilon,0}(x)u_0, \quad x\in\X.
\end{equation*}
Then $\sigma$ is (Lipschitz)-continuous and von Neumann-bounded, because for any $k\in\N$ and $x\in\X$ we clearly have
\begin{equation*}
    p_k(\sigma(x)) \leq p_k(u_+) + p_k(u_-) + p_k(u_0),
\end{equation*}
and the condition \eqref{eq: abstract condition on sigma} is satisfied.
\end{example}

The following theorem shows that a function $\sigma$ that satisfies Definition~\ref{sigmoid} is discriminatory, from which the density of $\mathfrak N(\sigma)$ follows by Theorem~\ref{prop: density}.%\fred{Add remark that this shows discriminatory property?}\nils{Ok now?}
\begin{theorem}\label{prop: discriminatory condition is satisfied}
Let $\X$ be a real Fr\'echet space. Let $\sigma:\X\to\X$ be continuous, von Neumann-bounded and satisfying the separating property in Definition~\ref{sigmoid} above. Assume that for a given compact subset $K\subset \X$ and a given regular Borel measure $\mu$ on $K$ it holds 
\begin{equation*}
    \int_K \langle \ell,\sigma(Ax+b)\rangle\, \mu(dx) = 0
\end{equation*}
for all $\ell\in\X',A\in\mathcal L(\X),b\in\X$. Then $\mu=0$.
\end{theorem}

Before we can prove Theorem~\ref{prop: discriminatory condition is satisfied} we need two preparatory lemmas. 
\begin{lemma}\label{lemma: trivial lemma}
Given $\phi,\psi \in \X'\setminus \{0\}$ there exists $z\in \X$ such that $\phi(z)=1$ and $\psi(z)\neq 0$.
\end{lemma}
\begin{proof}
Linearity of $\phi$ implies that the set $\Phi_+\cup \Phi_-$, where $\Phi_+ =\{ x\in\X; \langle\phi,x\rangle >0 \}$ and $\Phi_- =\{ x\in\X; \langle\phi,x\rangle <0 \}$, is actually dense. To see this, we need to show that each $x\in\Phi_0=\ker(\phi)$ can be approximated with a sequence in $\Phi_+\cup \Phi_-$. Consider $u_n= n^{-1}u\in\X$ with some $u\in \X$ such that $\phi(u)=1$ and define $x_n=x+u_n$. Then clearly $x_n \in \Phi_+$ and $x_n\to x$ and hence we get that $\text{cl}(\Phi_+\cup \Phi_-)=\X$.  Suppose that $\psi$ vanishes on the set $\Phi_+\cup \Phi_-$. 
Again by continuity of $\psi$ we would get $\psi=0$ identically. Therefore, there must exist $w\in \Phi_+\cup \Phi_-$ such that $\psi(w)\neq 0$. The element $z=w/\phi(w)$ does the job.
\end{proof}
The next lemma is crucial for the proof of Theorem \ref{prop: discriminatory condition is satisfied} as it allows us to rotate, shift and project to all possible half-spaces and show that the measures on them is zero if certain conditions are satisfied. 
\begin{lemma}\label{lemma: replication of linear functionals}
Let $\X$ be a real Fr\'echet space. Let $\psi\in \X'$ be not identically zero. Then, for arbitrary $\gamma\in \X'$, the equation
\begin{equation*}
    \gamma = \psi \circ A
\end{equation*}
is solvable for some $A\in\mathcal{L}(\X)$.
\end{lemma}
\begin{proof}
For arbitrary $\phi\in\X', t\in\R$ we write $\phi_t:=\{x\in\X; \langle \phi,x\rangle = t\}$. Clearly, we can assume $\gamma$ not identically zero, otherwise the problem is trivial. Therefore, let $z\in\X$ be such that $\langle \gamma,z\rangle =1$ and $\langle \psi,z\rangle \neq 0$. Clearly, such $z$ exists in view of Lemma~\ref{lemma: trivial lemma} above. Moreover, let $w\in \X$ such that $\langle \psi, w\rangle=1$. 

Let $\Psi_0=\ker(\psi)$ and $\Gamma_0=\ker(\gamma)$. We observe that
\begin{equation}
    \X = \Gamma_0 + \langle z \rangle = \Psi_0 + \langle w \rangle
\end{equation}
where $\langle z \rangle = \{sz; s\in\R\}\subset \X$ and  $\langle w \rangle = \{sw; s\in\R\}\subset \X$. Furthermore, $\Gamma_0\cap \langle z \rangle = \{0\}$ and $\Psi_0\cap \langle w \rangle = \{0\}$, namely $\Gamma_0$ and $\langle z \rangle$, are algebraic complements. The same holds for $\Psi_0$ and $\langle w \rangle$. Furthermore, $\Gamma_0$ and $\Psi_0$ are closed by continuity, and have codimension one. By Schaefer \cite[Prop. 3.5., page 22]{schaefer1971topological}, it follows that $\Gamma_0$ and $\langle z \rangle$ (respectively, $\Psi_0$ and $\langle w \rangle$) are also topologically complemented.

Therefore, any $x\in\X$ may be written in a unique way as 
\begin{equation*}
    x = x_{\Gamma_0} + \gamma(x)z = x_{\Psi_0} + \psi(x)w,
\end{equation*}
where $x_{\Gamma_0}\in \Gamma_0, x_{\Psi_0}\in \Psi_0$. %\fred{We have explicitly that $x_{\psi_0}=x-\psi(x)w$, which is the reason for (2)?} \luca{What do you mean?}\fred{The notation $x_{\psi_0}$ hides that this really is $x-\psi(x)w$, nothing more than that:-)}
We can therefore define the following projections operators:
\begin{equation*}
    \Pi_{\Gamma_0}:\X\to \Gamma_0, \quad x \mapsto x_{\Gamma_0},
\end{equation*}
\begin{equation*}
    \Pi_{\langle z \rangle}:\X\to \langle z \rangle\, \quad x \mapsto \gamma(x)z,
\end{equation*}
\begin{equation*}
    \Pi_{\Psi_0}:\X\to \Psi_0, \quad x \mapsto x_{\Psi_0},
\end{equation*}
\begin{equation*}
    \Pi_{\langle w \rangle}:\X\to \langle w \rangle\, \quad x \mapsto \psi(x)w.
\end{equation*}
Since $\psi$, $\gamma$ and the identity operator are continuous, it follows that  $\Pi_{\Psi_0}(x)=x-\psi(x)w$,  $\Pi_{\Gamma_0}(x)=x-\gamma(x)z$, $\Pi_{\langle z \rangle}$ and $\Pi_{\langle w \rangle}$ are in $\mathcal{L}(\X)$.
%By \cite[Prop 2.2, page 20]{schaefer1971topological}, we conclude that all these operators are elements of $\mathcal{L}(\X)$.
%\fred{Does this not follow directly from linearity of $\psi$ and its continuity? I mean, $\Pi_{\psi_0}(x)=x-\psi(x)w$, which is linear, and also continuous as identity is continuous and $\psi$ itself?} \luca{Oh, right! In this very specific case, we have an explicit expression for the projections, and hence we do not require abstract results...}
Define $A_0:= \Pi_{\Psi_0}\circ \Pi_{\Gamma_0} + \Pi_{\langle w \rangle }\circ \Pi_{\langle z \rangle}\in \mathcal{L}(\X)$. Let $x\in \X$ arbitrary, and write it as $x=  x_{\Gamma_0} + \gamma(x)z$. Write $z= z_{\Psi_0}+\psi(z)w$. Then,
\begin{equation*}
    A_0x = \Pi_{\Psi_0}x_{\Gamma_0} + \gamma(x)\Pi_{\langle w \rangle} z = \Pi_{\Psi_0}x_{\Gamma_0} + \gamma(x)\psi(z) w, 
\end{equation*}
and
\begin{equation*}
    \psi(A_0x) = \gamma(x)\psi(z) \psi(w) = \gamma(x)\psi(z), \quad x\in\X.
\end{equation*}
But $\psi(z)\neq 0$, and thus $A:= \psi(z)^{-1}A_0\in\mathcal{L}(\X)$ does the job.
\end{proof}
We are now ready to prove Theorem \ref{prop: discriminatory condition is satisfied}:

\begin{proof}[Proof of Theorem~\ref{prop: discriminatory condition is satisfied}]
Consider $\lambda>0$. Then for any $\ell\in\X',A\in\mathcal L(\X),b\in\X$ it holds 
\begin{equation*}
    \int_K \langle \ell,\sigma(\lambda (Ax+b))\rangle\, \mu(dx) = 0.
\end{equation*}
Observe that, as $\lambda\to \infty$, pointwise in $x\in\X$,
\begin{equation*}
     \langle \ell,\sigma(\lambda (Ax+b))\rangle \to
     \begin{cases}
     \langle \ell, u_+\rangle, \text{ if } Ax+b \in \Psi_+\\
      \langle \ell, u_-\rangle, \text{ if } Ax+b \in \Psi_-\\
       \langle \ell, u_0\rangle, \text{ if } Ax+b \in \Psi_0\\
     \end{cases}
\end{equation*}
Since, $\sigma$ is von Neumann-bounded, then there exists a constant $ C(\ell,\sigma)$ such that
\begin{equation*}
    \abs{\langle \ell,\sigma(\lambda (Ax+b))\rangle} \leq C(\ell,\sigma),
\end{equation*}
uniformly in $\lambda$ and $x$. By the Hahn-Jordan decomposition (see Bogachev \cite[Thm. 3.1.1., Cor. 3.1.2]{bogachev2007measureVol1}), we can write the measure $\mu=\mu_1-\mu_2$ for two positive measures $\mu_1,\mu_2$ on $K$.
This implies that $$\int_K \langle \ell,\sigma(\lambda (Ax+b))\rangle \mu (dx)=\int_K \langle \ell,\sigma(\lambda (Ax+b))\rangle \mu_1 (dx)-\int_K \langle \ell,\sigma(\lambda (Ax+b))\rangle \mu_2 (dx)$$
 
 Since we are integrating on the compact set $K$, and $\mu$ is a regular Borel measure, constants are integrable with respect to $\mu$ on $K$. The same holds then for $\mu_1$ and $\mu_2$. 
 
 Therefore, by Lebesgue's dominated convergence theorem applied to each integrand above, it follows that
\begin{equation}\label{eqn:zero:split:int}
     \langle \ell, u_+\rangle \mu[K\cap A^{-1}(\Psi_+ -b)] +
      \langle \ell, u_-\rangle \mu[K\cap A^{-1}(\Psi_- -b)] +
       \langle \ell, u_0\rangle \mu[K\cap A^{-1}(\Psi_0 -b) ] = 0
\end{equation}
for any $\ell\in\X',A\in\mathcal L(\X),b\in\X$. 

Let us first assume that $u_+ \notin \Span \{u_0,u_-\}$. Then by the Hahn-Banach theorem (see e.g. Conway \cite[Chap IV, Cor. 3.15]{conway2010}) we can choose $\ell\in\mathfrak X'$ such that $\langle \ell, u_+\rangle=1$ and  $\langle \ell, u_-\rangle=\langle \ell, u_0\rangle=0$. This leads us to conclude from (\ref{eqn:zero:split:int}) that $$\mu[K\cap A^{-1}(\Psi_+ -b)]=0$$
for all $A\in\mathcal L(\X),b\in\X$. Let now $t\in\R$ and $b\in\X$ such that $t=\psi(-b)$. Then, it is immediate to see that $$\Psi_+ -b=\psi^{-1}(t,\infty)$$ and thus $$\mu[K\cap (\psi\circ A)^{-1}(t,\infty)]=0$$
for each $t\in\R$ and $A\in\mathcal L(\X)$. By Lemma \ref{lemma: replication of linear functionals}, we therefore deduce that
\begin{equation}\label{eqn:right:interval}
\mu[K\cap \gamma^{-1}(t,\infty)]=0 
\end{equation}
for each $t\in\R$ and $\gamma\in \X'$. In the case that $u_- \notin \Span \{u_0,u_+ \}$ instead, a similar line of reasoning leads to conclude that 
\begin{equation}\label{eqn:left:interval}
\mu[K\cap \gamma^{-1}(-\infty,t)]=0.
\end{equation}
Observe in particular that $\mu(K)=0$. For the sake of convenience, we trivially extend $\mu$ to the whole $\X$, namely
\begin{equation*}
    \mu_{ext}(E) := \mu(K\cap E), \quad E\in\mathcal{B}(\X)
\end{equation*}
and notice that $\abs{\mu_{ext}}(\X)=\abs{\mu}(K)<\infty$, where 
$\abs{\mu_{ext}}= \mu_{ext,1} + \mu_{ext,2}$, and $\mu_{ext} = \mu_{ext,1} - \mu_{ext,2}$ is the Hahn-Jordan decomposition for the extended measure ($\mu_{ext,1}$ and $\mu_{ext,2}$ are positive finite measures on $\mathcal{B}(\X)$). Clearly, then it follows from $\mu(K)=0$ that  $\mu_{ext}(\X)=0$. Recall also that $\mathcal{B}(K)=\mathcal{B}(\X)\cap K$. 

Because $\mu$ is regular Borel measure, it follows in particular that for every $E\subset K$ and $\varepsilon >0$, there exists compact $K_{\varepsilon}\subset K$ such that $\abs{\mu}(E\setminus K_{\varepsilon})<\varepsilon$. This property extends to $E\in \X$ for $\mu_{ext}$ as we may use that $\abs{\mu_{ext}}(\cdot)=\abs{\mu}(\cdot\cap K)$ and choose $K_{\varepsilon} \subset E\cap K$ such that $\abs{\mu}((E\cap K)\setminus K_{\varepsilon})< \varepsilon $ and it follows that $\abs{\mu_{ext}}(E\setminus K_{\varepsilon})=\abs{\mu_{ext}}((E\cap K) \setminus K_{\varepsilon}) + \abs{\mu_{ext}}(E\cap K^c)=\abs{\mu}((E\cap K) \setminus K_{\varepsilon})<\varepsilon$. This shows that $\mu_{ext}$ is a Radon measure in the sense of \cite[Def. 7.1.1]{bogachev2007measure}. 

Moreover, \eqref{eqn:right:interval} or \eqref{eqn:left:interval} is now telling us that 
$
\mu_{ext} = 0$ on  $\sigma(\X')\subset\mathcal{B}(\X)$, 
the sigma-algebra generated by all the elements of $\X'$. We want to show that actually $\mu_{ext}=0$ on $\mathcal{B}(\X)$ as well. We argue by contradiction and assume there exists $E\in\mathcal{B}(\X)$ such that $\mu_{ext}(E)\neq 0$. In virtue of Bogachev \cite[Prop. 7.12.1]{bogachev2007measure} we may find $B\in\sigma(\X')$ such that
$$
\abs{\mu_{ext}}(E\Delta B) = 0,
$$
namely
$$
\mu_{ext,i}(E\Delta B) = 0, \quad i=1,2.
$$
Since $E\Delta B = (E\cup B)\setminus (E\cap B)$ and $\mu_{ext,i}$ are positive finite measures, we infer 
$$
\mu_{ext,i}(E\cup B)=\mu_{ext,i}(E\cap B), \quad i=1,2,
$$
which implies, $i=1,2$, 
$$
\begin{cases}
\mu_{ext,i}(E)\leq \mu_{ext,i}(E\cup B) = \mu_{ext,i}(E\cap B)\leq \mu_{ext,i}(E)\\
\mu_{ext,i}(B)\leq \mu_{ext,i}(E\cup B) = \mu_{ext,i}(E\cap B)\leq \mu_{ext,i}(B)
\end{cases}
$$
and finally $\mu_{ext,i}(E)=\mu_{ext,i}(B)$ for $i=1,2$. Therefore,
$$
0\neq \mu_{ext}(E) = \mu_{ext_1}(E)-\mu_{ext,2}(E) = \mu_{ext_1}(B)-\mu_{ext,2}(B) = \mu_{ext}(B)
$$
and at the same time $\mu_{ext}(B)=0$, because $B\in\sigma(\X')$. Thus, it must hold $\mu_{ext}=0$ on $\mathcal{B}(\X)$, and hence, $\mu=0$ on $\mathcal{B}(K)$, which concludes the proof.

\end{proof}

\subsection{Additional examples of functions with Separating property}\label{additional:examples}
We now provide a few more examples of function that satisfy the Separating property Definition~\ref{eq: abstract condition on sigma}. The first example resembles the well known rectified linear activation function (ReLU).
\begin{example}\label{example: identity map}
We consider the following example: let $(\X,\norm{\cdot})$ be a real Banach space now. Consider $\psi\in\X'$ with $\norm{\psi}=1$ (the dual norm). For $R>0$, let $B_R$ denote the open ball of radius $R$ around the origin. First of all we notice that 
\[
d(\operatorname{cl}(B_{R+1});\{x\in\X; \abs{\langle\psi,x\rangle}\geq R+2\}) \geq 1.
\]
Indeed, given $y:\norm{y}\leq R+1$ and $x: \abs{\langle\psi,x\rangle}\geq R+2$, it follows that $\norm{x}\geq R+2$ and thus
\[
\norm{y-x}\geq \abs{\norm{y}-\norm{x}}\geq 1.
\]
In particular these sets are disjoints.

Set $F_0:=\X\setminus B_{R+1}$ and $F_1:=\operatorname{cl}(B_R)$: these closed sets are disjoint. Since we are in a normal space, Urysohn's lemma ensures that there exists $\mathcal{U}:\X\to [0,1]$ continuous such that 
\[
\mathcal U\big|_{F_1}=1, \quad \mathcal U\big|_{F_0}=0.
\]
In particular, since $\{x\in\X;\abs{\langle\psi,x\rangle}\geq R+2\}\subset \X\setminus cl(B_{R+1})\subset F_0$, $\mathcal U = 0 $ on $\{x\in\X:\abs{\langle\psi,x\rangle}\geq R+2\}$.

Let $I_\geq$ and $I_\leq$ be the indicator functions of the sets $\{x\in\X: \langle\psi,x\rangle\geq R+2\}$ and $\{x\in\X:\langle\psi,x\rangle\leq -R-2\}$ respectively. And let $I_\geq^\varepsilon$ and $I_\leq^\varepsilon$ be their Lipschitz approximations, as in Example~\ref{ex: abstract condition on sigma}. Since, with the same notation as above, it holds
\[
\Supp I_\geq^\varepsilon = \operatorname{cl}(\{\psi\geq R+2\}_{\varepsilon}),\quad
\Supp I_\leq^\varepsilon = \operatorname{cl}(\{\psi\leq -R-2\}_{\varepsilon}),
\]
elementary computations show that
\[
\Supp I_\geq^\varepsilon \subset \X\setminus B_{R+2-\varepsilon},
\quad
\Supp I_\leq^\varepsilon \subset \X\setminus B_{R+2-\varepsilon}
\]
and thus for $\varepsilon<1$
\[
\Supp I_\geq^\varepsilon \cap \operatorname{cl}(B_{R+1})=\emptyset, \quad
\Supp I_\leq^\varepsilon \cap \operatorname{cl}(B_{R+1})=\emptyset.
\]
We can also easily get that 
\[
\Supp I_\geq^\varepsilon \subset \{\psi\geq R+2-\varepsilon\},\quad
\Supp I_\leq^\varepsilon \subset \{\psi\leq -R-2+\varepsilon\},
\]
showing that $\Supp I_\geq^\varepsilon \cap \Supp I_\leq^\varepsilon =\emptyset$.

We choose linearly independent vectors $u_\geq$ and $u_\leq$ and define
\[
\sigma(x) := \mathcal{U}(x)\,x + I_\geq^\varepsilon(x)\,u_\geq +  I_\leq^\varepsilon(x)\,u_\leq, \quad x\in \X.
\]
Then $\sigma\in C(\X;\X)$, and it is bounded because
\[
\norm{\sigma(x)} \leq R+1 + \norm{u_\geq} + \norm{u_\leq},\quad x\in\X.
\]
Clearly, $\sigma(x)=x$ if $\norm{x}\leq R$. 

Moreover, for $x\in\X$ such that $\langle \psi,x\rangle>0$, then for all $\lambda\geq \langle \psi,x\rangle^{-1}(R+2)$ we have $\sigma(\lambda x)=u_\geq$. Similarly, for $x\in\X$ such that $\langle \psi,x\rangle<0$, then for all $\lambda\geq \langle \psi,x\rangle^{-1}(-R-2)$ we have $\sigma(\lambda x)=u_\leq$. Finally, if $\langle\psi,x\rangle =0$, then for any $\lambda>0$ we have $\langle\psi,\lambda x\rangle=0$. Thus $\lambda x \notin \Supp I_\geq^\varepsilon \cup \Supp I_\leq^\varepsilon$ and so
\[
\sigma(\lambda x) = \mathcal{U}(\lambda x) \lambda x.
\]
If $x=0$, then $\sigma(\lambda x) = \sigma(0)=0$. If $x\neq 0$, then for all $\lambda$ larger than $\norm{x}^{-1}(R+1)$ it holds $\sigma(\lambda x)=0$.

This shows that $\sigma$ satisfies \eqref{eq: abstract condition on sigma}.
\end{example}

\begin{example}\label{ex: some concrete examples}
Let us give some further concrete applications of our abstract framework. Let now for the sake of simplicity $\X$ be a real separable Hilbert space  with inner product denoted by $\langle\cdot,\cdot\rangle$ and corresponding norm by $\Vert\cdot\Vert$. Further, we denote by $(e_k)_k$ an orthonormal basis for $\X$. Any $x\in\X$ may be uniquely written as $x=\sum_{k\in\N}x_ke_k$, where $x_k=\langle e_k,x\rangle$.

Consider $\beta_i\in C(\R;\R),i=1,2,3$ such that 
\begin{equation*}
    \begin{cases}
    \lim_{\xi\to\infty} \beta_1(\xi)=1,\;\lim_{\xi\to-\infty} \beta_1(\xi)=-1, \;\beta_1(0)=0,\\
    \lim_{\xi\to\infty} \beta_2(\xi)=1,\;\lim_{\xi\to-\infty} \beta_2(\xi)=1, \;\beta_2(0)=1,\\
    \lim_{\xi\to\infty} \beta_3(\xi)=-1,\;\lim_{\xi\to-\infty} \beta_3(\xi)=2, \;\beta_3(0)=0,
    \end{cases}
\end{equation*}
and define
\begin{equation*}
    \sigma(x)=\beta_1(x_1)e_1 + \beta_2(x_2)e_2 + \beta_3(x_1)e_3,\quad x\in\X.
\end{equation*}
Evidently, $\sigma\in C(\X;\X)$; besides, since $\norm{\sigma(x)}^2=\beta_1^2(x_1)+\beta_2(x_2)^2+\beta_3(x_1)^2$, it holds $\sup_{x}\norm{\sigma(x)}<\infty$, because $\beta_1,\beta_2$ and $\beta_3$ are bounded. Thus $\sigma$ is von Neumann-bounded. Consider now the linear bounded functional
\begin{equation*}
    \psi(x) := \langle e_1,x\rangle = x_1, \quad x\in\X.
\end{equation*}
Clearly, $\Psi_+=\{x\in\X;x_1>0\},\Psi_-=\{x\in\X;x_1<0\}$ and $\Psi_0=\{x\in\X;x_1=0\}$ and, as $\lambda\to\infty$ 
\begin{equation*}
\sigma(\lambda x) \to
\begin{cases}
 e_1+e_2-e_3, \text{ if } x\in \Psi_+\\
 -e_1+e_2+2e_3, \text{ if } x\in \Psi_-\\
 e_2, \text{ if } x\in \Psi_0
\end{cases}
\end{equation*}
which are linearly independent. We can therefore apply our results to infer that  $\mathfrak N(\sigma)$ is dense in $C(\X;\R)$ with respect to the topology of uniform convergence on the compact subsets of $\X$.

We can even go further. By the comment after Definition~\ref{sigmoid} indeed it is enough to consider a function $\beta\in C(\R;\R)$ such that 
\begin{equation*}
     \lim_{\xi\to\infty} \beta(\xi)=1,\;\lim_{\xi\to-\infty} \beta(\xi)=0, \;\beta(0)=0,
\end{equation*}
and arbitrary $z \in\X$ in order to define
\begin{equation*}
    \sigma(x) = \beta(\psi(x)) z=\beta(x_1) z,\quad x\in\X
\end{equation*}
which still enables us to conclude that $\mathfrak N(\sigma)$ is dense in $C(\X;\R)$. Example~\ref{ex: some concrete examples 2} below extends this example for more general choices of $\psi$. A natural question now would be to find ``optimal'' $\beta$ and $z$ such that the convergence of the approximation to the function we want to learn is ``fast''. 
\end{example}

\begin{example}\label{ex: infinite sum of activation functions}
The above example can be extended to an activation function that operates on infinitely many different directions $z_j\in\X$. More precisely, let now $\X$ be a real Banach space with norm denoted by $\norm{\cdot}$. As above, we consider an arbitrary $\psi\in\X'\setminus\{0\}$. Moreover, suppose we have a sequence $(\beta_j)_{j\in\N}\subset C(\R;\R)$ such that 
\begin{equation*}
    \lim_{\xi\to\infty} \beta_j(\xi)=1,\;\lim_{\xi\to-\infty} \beta_j(\xi)=0, \;\beta_j(0)=0,\; j\in\N
\end{equation*}
and $\sup_j\norm{\beta_j}_\infty=:B<\infty$.

Let $(z_j)_{j\in\N}\subset\X$ be such that $Z:=\sum_{j=1}^\infty\norm{z_j}<\infty$. Set
\[
z:=\sum_{j=1}^\infty z_j\in\X
\]
and assume $z\neq 0$.

We show that the map $\X\ni x\mapsto \sigma(x):=\sum_{j=1}^\infty \beta_j(\psi(x))z_j$ is an activation function.
\begin{enumerate}
    \item Well-defined: since it holds
    \[
    \sum_{j=1}^\infty\norm{\beta_j(\psi(x))z_j} = \sum_{j=1}^\infty\abs{\beta_j(\psi(x))}\norm{z_j}\leq
    \sum_{j=1}^\infty\norm{\beta_j}_\infty \norm{z_j}\leq BZ
    \]
    we have absolute convergence and so $\sigma(x)$ is well-defined.
    \item Boundedness: $\norm{\sigma(x)}\leq \sum_{j=1}^\infty\norm{\beta_j(\psi(x))z_j}\leq BZ$ for any $x\in\X$.
    \item Continuity: we have
    \[
    \norm{\sigma(x) - \sum_{j=1}^N\beta_j(\psi(x))z_j } =
     \norm{\sum_{j=N+1}^\infty\beta_j(\psi(x))z_j } \leq B \sum_{j=N+1}^\infty\norm{z_j},
    \]
    and thus
    \[
    \sup_{x\in\X} \norm{\sigma(x) - \sum_{j=1}^N\beta_j(\psi(x))z_j }\leq B \sum_{j=N+1}^\infty\norm{z_j}\to 0
    \]
    as $N\to\infty$, namely the convergence is uniform. Since $x\mapsto \sum_{j=1}^N\beta_j(\psi(x))z_j$ is continuous, $\sigma$ must be continuous as well.
    \item Separating property: Let $\lambda>0$. Consider first $x\in \Psi_+$. From the computations just done, we have
    \begin{equation*}
        \begin{split}
             \norm{\sigma(\lambda x) - z} & \leq
             \norm{\sigma(\lambda x) - \sum_{j=1}^N\beta_j(\psi(\lambda x)) z_j  } + 
             \norm{\sum_{j=1}^N\beta_j(\psi(\lambda x)) z_j -z } \\
             &\leq B\sum_{j=N+1}^\infty\norm{z_j} + \norm{\sum_{j=1}^N\beta_j(\psi(\lambda x)) z_j -z }.
        \end{split}
    \end{equation*}
   Fix $\varepsilon>0$ and chose $N_\varepsilon\in\N$ such that if $N\geq N_\varepsilon$ it holds $\sum_{j=N+1}^\infty\norm{z_j}\leq \frac{\varepsilon}{B}$. For such $N$ we have:
   \[
   \norm{\sigma(\lambda x) - z} \leq \varepsilon +
              \norm{\sum_{j=1}^N\beta_j(\psi(\lambda x)) z_j -z } 
   \]
   and thus 
   \[
   \limsup_{\lambda\to\infty} \norm{\sigma(\lambda x) - z} \leq \varepsilon +
              \norm{\sum_{j=1}^N z_j -z } 
   \]
   because evidently as $\lambda\to\infty$
   \[
   \sum_{j=1}^N\beta_j(\psi(\lambda x)) z_j -z \stackrel{\X}{\to}  \sum_{j=1}^N z_j -z .
   \]
   Hence
   \[
   \limsup_{\lambda\to\infty} \norm{\sigma(\lambda x) - z} \leq \varepsilon +
              \norm{\sum_{j=N+1}^\infty z_j } \leq \varepsilon + \frac{\varepsilon}{B}
   \]
   and by the arbitrariness of $\varepsilon$
   \[
   \lim_{\lambda\to\infty} \norm{\sigma(\lambda x) - z} = \limsup_{\lambda\to\infty} \norm{\sigma(\lambda x) - z} =0
   \]
   i.e. $\sigma(\lambda x)\to z$ as $\lambda\to \infty$, if $x\in \Psi_+$.
   
   The cases $x\in\Psi_-$ and $x\in\Psi_0$ are treated similarly (with $z=0$ now).
\end{enumerate}

\end{example}

\begin{example}
In view of the previous example, we further expand on the idea of an activation function operating on each coordinate. Let $\mathfrak{X}$ be a separable Hilbert space with an orthonormal basis $(e_k)_{k\in\mathbb N}$ and inner product naturally denoted $\langle\cdot,\cdot\rangle$. 
For $x\in\mathfrak{X}$, we define the activation function as
$$
\sigma(x)=\sum_{k=1}^{\infty}\hat{\sigma}(x_k)e_k
$$
where $x_k:=\langle x,e_k\rangle$ and $\hat{\sigma}:\mathbb R\rightarrow\mathbb R$. For a linear operator $A\in L(\mathfrak{X})$, we can introduce a family of linear functionals $\rho_k\in\mathfrak{X}'$ by
$$
\rho_k(x)=\langle Ax,e_k\rangle
$$
to obtain 
$$
Ax+b=\sum_{k=1}^{\infty}(\rho_k(x)+b_k)e_k
$$
with $b_k:=\langle b,e_k\rangle$. But then a neuron becomes, with $\ell=\sum_{k=1}^{\infty}\ell_ke_k\in\mathfrak{X}$,
\begin{equation}
\label{repr:infinite-wide-nn}
\langle\ell,\sigma(Ax+b)\rangle=\sum_{k=1}^{\infty}\ell_k\hat{\sigma}(\rho_k(x)+b_k)
\end{equation}
We remark that the representation on the right-hand side above links to infinite wide neural networks. Williams \cite{Williams} proposes and studies such networks using weighted integral representations of the infinite layer to encode the sum, and relates such networks to Gaussian processes (see also Cho and Saul \cite{ChoSaul}). As $\ell$ defines a linear functional, we can represent it as an integral operator rather than a sum which shows that our definition of neural networks is a generalisation of this class. Infinitely wide neural networks are based on the approximation results of Hornik \cite{Hornik-93}. 
%\fred{and maybe Sandberg-paper in our references as well??? Some more recent references, Jaarkola? Maybe include a line or two in introduction as well about infinite networks}.    

Observe that we must require $\hat{\sigma}(0)=0$, otherwise $\sigma(0)=\hat{\sigma}(0)\sum_{k=1}^{\infty}e_k\notin\mathfrak{X}$. Moreover, if $\hat{\sigma}$ is Lipschitz continuous, it follows readily that $\sigma$ becomes Lipschitz continuous. We have that,
$$
\vert\hat{\sigma}(x_k)\vert=\vert\hat{\sigma}(x_k)-\hat{\sigma}(0)\vert\leq K\vert x_k\vert
$$
and therefore $\sigma(x)\in\mathfrak{X}$ as
\begin{equation}
    \label{hilbert-activation-ex}
\sum_{k=1}^{\infty}\hat{\sigma}^2(x_k)\leq K^2\sum_{k=1}^{\infty}x_k^2<\infty.
\end{equation}
To stay within the framework developed in this paper, we also need to have a bounded activation function. However, in the infinite dimensional setting this does not come for free. In light of \eqref{hilbert-activation-ex} one could ask for an activation function $\hat{\sigma}$ which is bounded and goes sufficiently fast to zero around the origin. However, let $\hat{\sigma}=0$ on $[-\varepsilon,\varepsilon]$ with $0<\varepsilon<1$ say. Then, for $x\in \X$,
\[
\abs{\sigma(x)}^2 = \sum_{k: \,\abs{x_k}>\varepsilon} \abs{\hat{\sigma}(x_k)}^2
\]
If now $x=\sum_{k=1}^Ne_k$, then
\[
\abs{\sigma(x)}^2 = N\,\abs{\hat{\sigma}(1)}^2
\]
which blows up when $N$ grows. It is an interesting question to generalise our activation functions to go beyond boundedness and allow for linear or polynomial growth, say.
\end{example}

\section{Approximation for general codomain}\label{sec:BanachvaluedNN}

In this section we are going to show that our results can be extended to functions $f\in C(\X;\Y)$ where $(\Y,\norm{\cdot}_{\Y})$ is an $\mathbb F$-Banach space.

As a first step, we need the following simple lemma, which enables us to approximate with our neural network continuous functions from $\X$ into $\mathbb F^d, d\in\N $:
\begin{lemma}\label{lemma: NN valued in Fd}
Let $\X$ be an $\mathbb F$-Fr\'echet space, and let $\sigma:\X\to\X$ be continuous and discriminatory. Then, given $f\in C(\X;\mathbb F^d)$, a compact subset $K$ of $\X$, and $\varepsilon>0$, there exist $\mathcal{N}^i= \sum_{m=1}^M \alpha_m^i\mathcal{N}_{\ell_m^i,A_m^i,b_m^i}\in  \mathfrak N(\sigma)$, $i=1,\dots ,d$, with suitable $\alpha_m^i\in\mathbb F, \ell_m^i\in\X',A_m^i\in\mathcal{L}(\X)$ and $b_m^i\in\X$ such that
\begin{equation*}\label{eq: NN valued in Fd}
   \sup_{x\in K}\norm{f(x) - (\mathcal{N}^1(x),\dots, \mathcal{N}^d(x))}_{\mathbb{F}^d} < \varepsilon
\end{equation*}
where for all $\xi\in \mathbb F^d$ we have $\norm{\xi}_{\mathbb F^d} = \sum_{i=1}^d\abs{\xi^i}$.
\end{lemma}
\begin{proof}
We write $f=(f^1,\dots ,f^d)$ with $f^i\in C(\X;\mathbb F), i=1,\dots, d$. Given $K\subset  \X$ and $\varepsilon>0$, Theorem \ref{approx:prop} guarantees the existence of $\mathcal{N}^i\in  \mathfrak N(\sigma)$ such that 
$$
\sup_{x\in K}\abs{f^i(x) - \mathcal{N}^i(x)} < \varepsilon/d
$$
and we are done.
\end{proof}
We are now ready to prove the following:
\begin{theorem}\label{thm: NN valued in Banach}
Let $\X$ be an $\mathbb F$-Fr\'echet space, and let $\sigma:\X\to\X$ be continuous and discriminatory. Let $(\Y,\norm{\cdot}_{\Y})$ be an $\mathbb F$-Banach space. Then, given $f\in C(\X;\Y)$, a compact subset $K$ of $\X$, and $\varepsilon>0$, there exist $d\in\N$, $v_1,\dots ,v_d$ linear independent unit vectors of $\Y$, $\mathcal{N}^1, \dots \mathcal{N}^d\in  \mathfrak N(\sigma)$, such that, by defining
$$
\mathcal{N}(x) := \sum_{i=1}^d \mathcal{N}^i(x)v_i,\quad x\in\X,
$$
it holds
\begin{equation*}\label{eq: NN valued in Banach}
   \sup_{x\in K}\norm{f(x) - \mathcal{N}(x)}_{\Y} < \varepsilon.
\end{equation*}
\end{theorem}
\begin{proof}
We recall the following general approximation result (see for example Brezis \cite[Ch. 6.1]{brezis2010functional}): given a topological space $(Z,\tau)$, an $\mathbb F$-Banach space $(\Y,\norm{\cdot}_{\Y})$ and a continuous map
$$
T: Z\to \Y
$$
such that $T(Z)$ is relatively compact in $Y$, then, given $\varepsilon>0$ there exists $T_\varepsilon: Z\to \Y$ continuous, with $T_\varepsilon(Z)$ contained in a finite-dimensional subspace of $\Y$, and such that
$$
\norm{T_\varepsilon(z) - T(z)}_{\Y}<\varepsilon, \quad z\in Z.
$$
To apply this result in our present setting, we first restrict $f$ to $K$
$$
f\big |_K : K \to \Y,
$$
obtaining a continuous function whose range is compact in $\Y$. Therefore, we may find $f_\varepsilon:K\to \Y$ continuous and such that
\begin{enumerate}
    \item $f_\varepsilon(K)\subset \Span\{v_1,\dots v_d\}\subset \Y$ for suitable linear independent elements $v_1,\dots ,v_d$, whose norm we assume to be equal to 1.
    \item $ \sup_{x\in K}\norm{f(x)-f_\varepsilon(x)}_{\Y} = \sup_{x\in K}\norm{f\big |_K(x)-f_\varepsilon(x)}_{\Y} < \varepsilon/2 $.
\end{enumerate}
We set for convenience $V =  \Span\{v_1,\dots v_d\} $, and we write $f_\varepsilon$ as 
$$
f_\varepsilon(x) = \sum_{i=1}^d f^i_\varepsilon(x) v_i, \quad x\in K
$$
with suitable $f^i_\varepsilon\in C(K;\mathbb F)$, $i=1,\dots, d$. Being $\X$ metrizable, it is clearly normal. Therefore, by the Tietze extension theorem (since $K$ is closed), there exist $g^i_\varepsilon\in C(\X;\mathbb F)$ extensions of $f^i_\varepsilon$, $i=1,\dots, d$. 

We define $g_\varepsilon(x):=\sum_{i=1}^dg^i_\varepsilon(x)v_i,\,x\in\X$. Then $g_\varepsilon\in C(\X;\Y)$, $g_\varepsilon(\X)\subset V $ and 
$$
 \sup_{x\in K}\norm{f(x)-g_\varepsilon(x)}_{\Y}  < \varepsilon/2.
$$
By Lemma \ref{lemma: NN valued in Fd} we may approximate on $K$
$$
\X\ni x \mapsto (g^1_\varepsilon(x),\dots ,g^d_\varepsilon(x))\in \mathbb F^d
$$
with $(\mathcal{N}^1,\dots,\mathcal{N}^d)$ such that
$$
\sup_{x\in K} \norm{ (g^1_\varepsilon(x),\dots ,g^d_\varepsilon(x)) -  (\mathcal{N}^1(x),\dots,\mathcal{N}^d(x)) }_{\mathbb F^d} < \varepsilon/2.
$$
We define 
$$
\mathcal{N}(x) := \sum_{i=1}^d \mathcal{N}^i(x)v_i,\quad x\in\X,
$$
which has the required property, since we have
\begin{equation*}
    \begin{split}
     \sup_{x\in K}\norm{f(x) - \mathcal{N}(x)}_{\Y} & \leq 
          \sup_{x\in K}\norm{f(x) - g_\varepsilon(x)}_{\Y} 
      + \sup_{x\in K}\norm{g_\varepsilon(x) - \mathcal{N}(x)}_{\Y} \\
      & < \varepsilon/2 + \sup_{x\in K}\sum_{i=1}^d\abs{g^i_\varepsilon(x) - \mathcal{N}^i(x)}\,\norm{v_i}_{\Y}\\
      & = \varepsilon/2 + \sup_{x\in K}\sum_{i=1}^d\abs{g^i_\varepsilon(x) - \mathcal{N}^i(x)}\\
       & = \varepsilon/2 + \sup_{x\in K} \norm{ (g^1_\varepsilon(x),\dots ,g^d_\varepsilon(x)) -  (\mathcal{N}^1(x),\dots,\mathcal{N}^d(x)) }_{\mathbb F^d}\\
      & < \varepsilon.
    \end{split}
\end{equation*}

\end{proof}

\section{Approximation with finite dimensional neural networks}\label{finte:approx:section}

In this section we prove a result that ensures that one can approximate a given abstract neural net arbitrary well via a neural network that is constructed from finite dimensional maps and can thus be trained. Of course, this can only work if we can approximate any given $x\in \X$ sufficiently well with a finite dimensional quantity as otherwise we could not even represent $x$ in a computer. It is therefore plausible that we can derive such results only if some kind of approximation property holds on $\X$. This approximation property must ensure that one can approximate the identity map on $\X$ by continuous linear maps of finite rank, uniformly on some subset $K\subset \X$ of interest. In spaces with a countable Schauder basis $(e_n)_{n\in \mathbb{N}}$, the approximating linear maps are usually the projections $\Pi_N: \X \to \Span\{e_1,\dots ,e_N\}$. Unfortunately, not every Fr\'echet space has a Schauder basis as shown by Enflo \cite{10.1007/BF02392270}. We refer the reader to Schaefer \cite[Ch. III, Sec. 9]{schaefer1971topological} for a discussion of the approximation property and existence of a Schauder basis for Fr\'echet space, which was an open problem until answered in \cite{10.1007/BF02392270}. Whenever the space $\X$ has a Schauder basis, however, we can actually derive an approximation of our abstract neural network with a trainable finite dimensional neural network as we shall see in this section. 

To start, we are first going to work in a Banach space setting.  %\nils{Hida space was mentioned here previously. We should maybe give a few more non Banach examples later when we extend.}%An essential ingredient of the proof is the so called approximation property, which may hold even for more general spaces (Hida for instance?)
Let therefore $\X$ be a real separable Banach space with norm denoted by $\Vert\cdot\Vert$ that admits a normalized Schauder basis $(e_k)_{k\in\N}$, namely each $x\in\X$ has a unique representation $x=\sum_{k=1}^\infty x_k e_k$ and $\norm{e_k}=1$ for all $k$.
It follows as in Schaefer \cite[Thm. 9.6, p. 115]{schaefer1971topological} that 
\begin{equation*}
    \Pi_N: \X \to \Span\{e_1,\dots ,e_N\},\quad x\mapsto \sum_{k=1}^N x_k e_k,\quad N\in\N
\end{equation*}
is linear and bounded with $\sup_{N\in\N}\norm{\Pi_N}_{op}\leq C$ for some suitable constant $C\geq 1$, and that for any $K\subset\X$ compact we have $\sup_{x\in K}\norm{x-\Pi_N x}\to 0$ as $N\to\infty$.

While we know by \cite{10.1007/BF02392270} that there exist Banach spaces without a Schauder basis, it is also true that ``all usual separable Banach spaces of Analysis admit a Schauder basis'' (see Brezis \cite{brezis2010functional}). 
For example for the Banach spaces $L^p(\mathbb{R}^n)$, where $1\leq p <\infty$, as well as for the Sobolev and Besov spaces, a basis is given by wavelets (see Triebel \cite{HansTriebel2004}). See Heil \cite{heil2011basis} for many more examples.

We assume now that the activation function $\sigma:\X\to\X$ is Lipschitz, namely 
\begin{equation}\label{eq: Lipschitz 1}
    \norm{\sigma(x)-\sigma(y)}\leq \text{Lip}(\sigma) \norm{x-y}, \quad x,y\in\X. 
\end{equation}
where $0\leq \text{Lip}(\sigma)<\infty$. Of course since $\X$ is already a metric space, we do not use the metric $d$ defined in (\ref{metric:loc:conv:space}), but the one implied by the norm, i.e. $d(x_1,x_2)=\norm{x_1-x_2}$.

Observe also that the activation functions in Example \ref{ex: some concrete examples} become Lipschitz as soon as we impose that the $\beta_i$'s are Lipschitz. The activation function in Example \ref{ex: abstract condition on sigma} is already Lipschitz. Therefore, this condition does not seem very restrictive.

We are ready to prove:

\begin{proposition}\label{prop: finite dimensional approx, Banach}
Let $\X$ be a real separable Banach space that admits a normalized Schauder basis $(e_k)_{k\in\N}$ and let $\sigma$ be Lipschitz. Let $f\in C(\X;\R)$, $K\subset\X$ compact and $\varepsilon>0$. Assume
\begin{equation*}
    \mathcal N^{\epsilon} (x) = \sum_{j=1}^M\langle \ell_j,\sigma(A_jx+b_j)\rangle,\quad x\in\X
\end{equation*}
with $\ell_j\in\X',A_j\in\mathcal{L}(\X)$ and $b_j\in\X$ such that 
\begin{equation*}
    \sup_{x\in K}\abs{f(x)-\mathcal N^{\epsilon}(x)}<\varepsilon.
\end{equation*}
Fix $\delta>0$. Then there exists $N_\ast=N_\ast(\mathcal N^\epsilon,\delta)\in\N$ such that for $N\geq N_\ast$
\begin{equation}\label{approx:finite}
    \sup_{x\in K}\abs{f(x)-\sum_{j=1}^M\langle \ell_j\circ\Pi_N,\sigma(
    \Pi_{N}A_j\Pi_{N}x+\Pi_{N}b_j)\rangle}<\varepsilon+\delta.
\end{equation}
\end{proposition}
\begin{proof}
For $j=1,\dots,M$, $N\in\N$ and $x\in K$ we indeed have
\begin{equation*}
    \begin{split}
        &\abs{\langle \ell_j,\sigma(A_jx+b_j)\rangle-\langle\ell_j\circ\Pi_N,\sigma(
    \Pi_{N}A_j\Pi_{N}x+\Pi_{N}b_j)\rangle}  \\
    &\qquad\qquad \leq 
    \abs{\langle \ell_j,\sigma(A_jx+b_j) - \Pi_N\sigma(A_jx+b_j)\rangle} \\
    &\qquad \qquad \qquad+
    \abs{\langle \ell_j,\Pi_N\sigma(A_jx+b_j)-\Pi_N\sigma(
    \Pi_{N}A_j\Pi_{N}x+\Pi_{N}b_j)\rangle}\\
    &\qquad \qquad\leq \norm{\ell_j} \norm{\sigma(A_jx+b_j) - \Pi_N\sigma(A_jx+b_j)}\\           
    &\qquad \qquad \qquad+ 
    \norm{\ell_j} C \norm{ \sigma(A_jx+b_j)-\sigma( \Pi_{N}A_j\Pi_{N}x+\Pi_{N}b_j) },
    \end{split}
\end{equation*}
where in the last line we have used that $\sup_{N\in\N}\norm{\Pi_N}_{op}\leq C$. Thus, as far as it concerns the second term, it holds
\begin{equation*}
    \begin{split}
        &\norm{\ell_j} C \norm{ \sigma(A_jx+b_j)-\sigma( \Pi_{N}A_j\Pi_{N}x+\Pi_{N}b_j) }  \\
    &\qquad\qquad \leq
    C\norm{\ell_j}\text{Lip}(\sigma)\norm{A_jx+b_j- \Pi_{N}A_j\Pi_{N}x-\Pi_{N}b_j}\\
     &\qquad\qquad \leq
     C\norm{\ell_j}\text{Lip}(\sigma)
     \left\{
     \norm{A_jx - \Pi_NA_jx} + \norm{\Pi_NA_jx-\Pi_NA_j\Pi_Nx} + \norm{b_j-\Pi_Nb_j}     
     \right\}\\
     &\qquad\qquad \leq
     C\norm{\ell_j}\text{Lip}(\sigma)
     \left\{
     \norm{A_jx - \Pi_NA_jx} + C\norm{A_jx-A_j\Pi_Nx} + \norm{b_j-\Pi_Nb_j}     
     \right\}\\
     &\qquad\qquad \leq
     C\norm{\ell_j}\text{Lip}(\sigma)
     \left\{
     \norm{A_jx - \Pi_NA_jx} + C\norm{A_j}_{op}\norm{x-\Pi_Nx} + \norm{b_j-\Pi_Nb_j}     
     \right\}\\
      &\qquad\qquad \leq
     C\norm{\ell_j}\text{Lip}(\sigma)
     \left\{
     \sup_{x\in K}\norm{A_jx - \Pi_NA_jx} \right.\\
    &\quad\quad\quad\quad\quad\quad\quad\quad\quad\quad\left. + C\norm{A_j}_{op}\sup_{x\in K}\norm{x-\Pi_Nx} + \norm{b_j-\Pi_Nb_j}     
     \right\}\\
     &\qquad\qquad =
     C\norm{\ell_j}\text{Lip}(\sigma)
     \left\{
     \sup_{y\in A_jK}\norm{y - \Pi_Ny} \right.\\
    &\quad\quad\quad\quad\quad\quad\quad\quad\quad\quad\left. + C\norm{A_j}_{op}\sup_{x\in K}\norm{x-\Pi_Nx} + \norm{b_j-\Pi_Nb_j}     
     \right\}.
    \end{split}
\end{equation*} 
Setting for convenience $\sigma_j := \sigma(A_jK +b_j)$, and noticing that it is compact, we eventually arrive at
\begin{equation*}
    \begin{split}
        &\abs{\langle \ell_j,\sigma(A_jx+b_j)\rangle-\langle\ell_j\circ \Pi_N,\sigma(
    \Pi_{N}A_j\Pi_{N}x+\Pi_{N}b_j)\rangle}  \\
     &\qquad\qquad \leq
     \norm{\ell_j}\text{Lip}(\sigma)
     \left\{
     \sup_{y\in A_jK}\norm{y - \Pi_Ny} \right.\\
    &\quad\quad\quad\quad\quad\quad\quad\quad\quad\quad\left. + C\norm{A_j}_{op}\sup_{x\in K}\norm{x-\Pi_Nx} + \norm{b_j-\Pi_Nb_j}     
     \right\}\\
     &\qquad \qquad + 
     \norm{\ell_j} \sup_{y\in\sigma_j}\norm{y-\Pi_Ny}
      \end{split}
\end{equation*}
Observe that $A_jK\subset\X$ is compact. By the approximation property provided by the Schauder basis $(e_k)_{k\in\N}$, we may find $N(j)\in\N$ such that:
\begin{equation*}
    \begin{cases}
    \sup_{y\in A_jK}\norm{y - \Pi_Ny} < \frac{\delta}{4M\norm{\ell_j}\text{Lip}(\sigma)}\\
     \sup_{y\in\sigma_j}\norm{y-\Pi_Ny} < \frac{\delta}{4M\norm{\ell_j}}\\
    \sup_{x\in K}\norm{x-\Pi_Nx} < \frac{\delta}{4M\norm{\ell_j}C\norm{A_j}_{op} \text{Lip}(\sigma)},\quad \text{if } \norm{A_j}_{op}\neq 0\\
    \norm{b_j-\Pi_Nb_j}   < \frac{\delta}{4M\norm{\ell_j}\text{Lip}(\sigma)}
    \end{cases}
\end{equation*}
for all $N\geq N(j)$. With this choice, we then have
\begin{equation*}
    \sup_{x\in K}
    \abs{\langle \ell_j,\sigma(A_jx+b_j)\rangle-\langle\ell_j\circ\Pi_N,\sigma(
    \Pi_{N}A_j\Pi_{N}x+\Pi_{N}b_j)\rangle}
    <\delta/M. 
\end{equation*}
Therefore, setting $N_\ast:=\max\{N(1),\dots,N(M)\}$, we conclude that for all $N\geq N_\ast$
\begin{equation*}
    \sup_{x\in K}\abs{f(x)-\sum_{j=1}^M\langle \ell_j\circ\Pi_N,\sigma(
    \Pi_{N}A_j\Pi_{N}x+\Pi_{N}b_j)\rangle}<\varepsilon+\delta.
\end{equation*}
\end{proof}
We mention that the function $\mathcal{N}^{\varepsilon}: \X\rightarrow \R$, which is required in the proposition above, exists for instance in view of Theorem \ref{prop: density}, as soon as we assume that $\sigma$ is discriminatory.
\begin{remark}
The terms appearing in the sum in (\ref{approx:finite}) can now easily be programmed in a computer. We see that for large $N$, it is sufficient to consider the finite dimensional input values $\Pi_N (x)$ instead of $x$, and then successively the restriction of the operators $\Pi_N A_j, \sigma$ and $\ell_j$ to $\Span\{e_1,\dots ,e_N\}$ instead of the maps $A_j, \sigma$ and $\ell_j$ for $j=1,\dots , M$. The maps $\Pi_N A_j, \sigma$ and $\ell_j$ are finite dimensional when restricted to $\Span\{e_1,\dots ,e_N\}$ and the sum above thus resembles a classical neural network. However, instead of the typical one dimensional activation function, the function $\Pi_N \circ \sigma$ restricted to $\Span\{e_1,\dots ,e_N\}$ is multidimensional. 
\end{remark}
With an extra effort it is possible to generalize this result to real separable Fr\'echet spaces that admit Schauder basis. Examples include for instance the Schwartz space of rapidly decreasing functions, for which a basis is given in terms of Hermite functions (see Schwartz \cite{schwartz1957theorie}) and the Hida test function and distribution space (see Holden {\it et al.} \cite[Def 2.3.2.]{holden2010stochastic}).

Let us now see how to do this generalization. Following Meise and Vogt \cite[28.10, p. 331]{meise1992einfuhrung}, a Schauder basis for a real separable Fr\'echet space is a sequence $(e_k)_{k\in\N}\subset\X$, such that each $x\in\X$ has a unique representation $x=\sum_{k=1}^\infty x_k e_k$. As above, we define 
\begin{equation*}
    \Pi_N: \X \to \Span\{e_1,\dots ,e_N\},\quad x\mapsto \sum_{k=1}^N x_k e_k,\quad N\in\N
\end{equation*}
which is linear and bounded. Still from Meise and Vogt \cite[28.10, p. 331]{meise1992einfuhrung}, we see that for any $j\in\N$ there exists $m\in\N$ and $C>0$ such that for any $x\in\X$ 
\begin{equation}\label{eq: uniform bound for the projections}
    \sup_{N\in\N} p_j\left(\Pi_Nx\right) \leq C p_m(x)
\end{equation}
Moreover, we can easily see that for any $K\subset\X$ compact and any $j\in\N$ we have 
$$
\sup_{x\in K}p_j(x-\Pi_N x)\to 0
$$ 
as $N\to\infty$. Indeed, following Schaefer \cite[p. 81]{schaefer1971topological} and from \eqref{eq: uniform bound for the projections} we see that
\begin{equation*}
     \sup_{N\in\N} \sup_{x\in S} p_j\left(\Pi_Nx\right) \leq C \sup_{x\in S}p_m(x)<\infty
\end{equation*}
for any $j\in\N$ and $S\subset\X$ with finite cardinality. Trivially, $\sup_{x\in S} p_j\left(x\right)<\infty$. We therefore deduce that the subset $\{\Pi_N\}_N \cup \{I\}\subset \mathcal L (\X)$ is simply bounded, with $I$ being the identity map. By Schaefer \cite[Thm 4.2, p. 83]{schaefer1971topological}, it is equicontinuous, being $\X$ a Baire space. By Schaefer \cite[Thm 4.5, p. 85]{schaefer1971topological} we therefore conclude that we have convergence on all precompact subsets of $\X$.

We are now going to impose the following ``graded'' Lipschitz condition on the non-linearity $\sigma$:
\begin{equation}\label{eq: Lipschitz 2}
    \exists k_0\in\N: \forall k\geq k_0 \;\exists C_k\geq 0: p_k(\sigma(x)-\sigma(y))\leq C_k p_k(x-y),\quad x,y\in\X.
\end{equation}
Notice that such a map $\sigma$ is automatically continuous.

We are ready to prove:
\begin{theorem}\label{prop: finite dimensional approx, Frechet}
Let $\X$ be a real separable Fr\'echet space that admits a Schauder basis $(e_k)_{k\in\N}$ and let $\sigma$ satisfy condition \eqref{eq: Lipschitz 2}. Let $f\in C(\X;\R)$, $K\subset\X$ compact and $\varepsilon>0$. Assume
\begin{equation*}
    \mathcal{N}^{\varepsilon}(x) = \sum_{j=1}^M\langle \ell_j,\sigma(A_jx+b_j)\rangle,\quad x\in\X
\end{equation*}
with $\ell_j\in\X',A_j\in\mathcal{L}(\X)$ and $b_j\in\X$ such that 
\begin{equation*}
    \sup_{x\in K}\abs{f(x)-\mathcal{N}^{\varepsilon}(x)}<\varepsilon.
\end{equation*}
%\fred{I do not like adding the previous sentence within the theorem-statement. It should be better outside in a remark. Same goes for Prop. 3.1! Maybe also use different notation than $f_{\sigma}$ to follow the neural net-notation, $\mathcal N$?}
Fix $\delta>0$. Then there exists $N_\ast=N_\ast(\mathcal N^\epsilon,\delta)\in\N$ such that for $N\geq N_\ast$
\begin{equation*}
    \sup_{x\in K}\abs{f(x)-\sum_{j=1}^M\langle \ell_j\circ\Pi_N,\sigma(
    \Pi_{N}A_j\Pi_{N}x+\Pi_{N}b_j)\rangle}<\varepsilon+\delta.
\end{equation*}
\end{theorem}
\begin{proof}
For $j=1,\dots,M$, $N\in\N$ and $x\in K$ we indeed have, for suitable integers $r(\ell_j)$, $t(\ell_j),m(\ell_j,\sigma)$ and $n(\ell_j,\sigma,A_j)$,
\begin{align*}
    %\begin{split}
        &\abs{\langle \ell_j,\sigma(A_jx+b_j)\rangle-\langle\ell_j\circ\Pi_N,\sigma(
    \Pi_{N}A_j\Pi_{N}x+\Pi_{N}b_j)\rangle}  \\
    &\qquad\qquad \leq
    \abs{\langle \ell_j,\sigma(A_j x+b_j) -\Pi_N\sigma(A_j x + b_j) \rangle }\\
     &\qquad \qquad \qquad+
    \abs{\langle \ell_j,\Pi_N\sigma(A_j x+b_j) -\Pi_N\sigma(\Pi_N A_j \Pi_N x + \Pi_N b_j) \rangle }\\
    &\qquad\qquad \leq
    C(\ell_j)p_{r(\ell_j)} ( \sigma(A_j x+b_j) -\Pi_N\sigma(A_j x + b_j)) \\   
    &\qquad\qquad\qquad +
    C(\ell_j)p_{r(\ell_j)} (\Pi_N\sigma(A_j x+b_j) -\Pi_N\sigma(\Pi_N A_j \Pi_N x + \Pi_N b_j) )\\
    &\qquad\qquad \leq
    C(\ell_j)p_{r(\ell_j)} ( \sigma(A_j x+b_j) -\Pi_N\sigma(A_j x + b_j)) \\   
    &\qquad\qquad\qquad +
    C(\ell_j)\sup_{N\in\N}p_{r(\ell_j)} (\Pi_N\sigma(A_j x+b_j) -\Pi_N\sigma(\Pi_N A_j \Pi_N x + \Pi_N b_j) )\\
    &\qquad\qquad \leq
    C(\ell_j)p_{r(\ell_j)} ( \sigma(A_j x+b_j) -\Pi_N\sigma(A_j x + b_j)) \\   
    &\qquad\qquad\qquad +
    C(\ell_j)C p_{t(\ell_j)} (\sigma(A_j x+b_j) -\sigma(\Pi_N A_j \Pi_N x + \Pi_N b_j) ),
    %\end{split}
\end{align*}
where in the last line we have used the fact that the constant $C$ in \eqref{eq: uniform bound for the projections} is independent of $N$ and $x$. Therefore, for the second term in the last expression we have
\begin{align*}
    %\begin{split}
        &C(\ell_j)\,p_{t(\ell_j)}(\sigma(A_jx+b_j)- \sigma(\Pi_{N}A_j\Pi_{N}x+\Pi_{N}b_j))\\
     &\quad \leq
    C(\ell_j)\,p_{t(\ell_j)\vee k_0}(\sigma(A_jx+b_j)- \sigma(\Pi_{N}A_j\Pi_{N}x+\Pi_{N}b_j))\\
    &\quad \leq
    C(\ell_j)C_{t(\ell_j)\vee k_0}\,p_{t(\ell_j)\vee k_0}(A_jx+b_j- \Pi_{N}A_j\Pi_{N}x-\Pi_{N}b_j)\\
    &\quad \leq
    C(\ell_j,\sigma)\left\{
    p_{t(\ell_j)\vee k_0}(A_jx- \Pi_{N}A_jx)+
    p_{t(\ell_j)\vee k_0}(\Pi_NA_jx- \Pi_{N}A_j\Pi_{N}x)\right.\\
    &\quad\quad\quad\quad\quad\quad\quad\quad\left. +p_{t(\ell_j)\vee k_0}(b_j-\Pi_{N}b_j)
    \right\}\\
    &\quad \leq
    C(\ell_j,\sigma)\left\{
    p_{t(\ell_j)\vee k_0}(A_jx- \Pi_{N}A_jx)+
    C'(\ell_j,\sigma)
    p_{m(\ell_j,\sigma)}(A_jx-A_j\Pi_{N}x)\right.\\
    &\quad\quad\quad\quad\quad\quad\quad\quad\left. +p_{t(\ell_j)\vee k_0}(b_j-\Pi_{N}b_j)
    \right\}\\
    &\quad \leq
    C(\ell_j,\sigma)\left\{
    p_{t(\ell_j)\vee k_0}(A_jx- \Pi_{N}A_jx)+
    C'(\ell_j,\sigma,A_j)
    p_{n(\ell_j,\sigma,A_j)}(x-\Pi_{N}x)\right.\\
    &\quad\quad\quad\quad\quad\quad\quad\quad\left. +p_{t(\ell_j)\vee k_0}(b_j-\Pi_{N}b_j)
    \right\}\\
    &\quad \leq
    C(\ell_j,\sigma)\left\{
    \sup_{x\in K}p_{t(\ell_j)\vee k_0}(A_jx- \Pi_{N}A_jx)\right.\\
    &\quad\quad\quad\quad\quad\quad\quad\quad\left.
    +C'(\ell_j,\sigma,A_j)\sup_{x\in K}  p_{n(\ell_j,\sigma,A_j)}(x-\Pi_{N}x) +p_{t(\ell_j)\vee k_0}(b_j-\Pi_{N}b_j)
    \right\}\\
    &\quad \leq
    C(\ell_j,\sigma)\left\{
    \sup_{y\in A_jK}p_{t(\ell_j)\vee k_0}(y- \Pi_{N}y)\right.\\
    &\quad\quad\quad\quad\quad\quad\quad\quad\left.
    +C'(\ell_j,\sigma,A_j)\sup_{x\in K}  p_{n(\ell_j,\sigma,A_j)}(x-\Pi_{N}x) +p_{t(\ell_j)\vee k_0}(b_j-\Pi_{N}b_j)
    \right\}.
    %\end{split}
\end{align*}
Observe that $A_jK\subset\X$ is compact. Setting for convenience $\sigma_j := \sigma(A_jK +b_j)$, and noticing that it is compact, we eventually arrive at
\begin{equation*}
    \begin{split}
        &\abs{\langle \ell_j,\sigma(A_jx+b_j)\rangle-\langle\ell_j\circ\Pi_N,\sigma(
    \Pi_{N}A_j\Pi_{N}x+\Pi_{N}b_j)\rangle}  \\
       &\quad \leq
    C(\ell_j,\sigma)\left\{
    \sup_{y\in A_jK}p_{t(\ell_j)\vee k_0}(y- \Pi_{N}y) + 
    \sup_{y\in\sigma_j}p_{r(\ell_j)}(y-\Pi_Ny)
    \right.\\
    &\quad\quad\quad\quad\quad\quad\quad\quad\left.
    +C'(\ell_j,\sigma,A_j)\sup_{x\in K}  p_{n(\ell_j,\sigma,A_j)}(x-\Pi_{N}x) +p_{t(\ell_j)\vee k_0}(b_j-\Pi_{N}b_j)
    \right\}.
    \end{split}
\end{equation*}

By the approximation property provided by the Schauder basis $(e_k)_{k\in\N}$, we may find $N(j)\in\N$ such that:
\begin{equation*}
    \begin{cases}
     \sup_{y\in A_jK}p_{t(\ell_j)\vee k_0}(y- \Pi_{N}y)< \frac{\delta}{4M C(\ell_j,\sigma)}\\
     \sup_{y\in\sigma_j}p_{r(\ell_j)}(y-\Pi_Ny) < \frac{\delta}{4M C(\ell_j,\sigma)}\\
     \sup_{x\in K}  p_{n(\ell_j,\sigma,A_j)}(x-\Pi_{N}x)  < \frac{\delta}{4MC(\ell_j,\sigma)C'(\ell_j,\sigma,A_j)},\quad \text{if } A_j\neq 0\\
    p_{t(\ell_j)\vee k_0}(b_j-\Pi_{N}b_j) < \frac{\delta}{4M C(\ell_j,\sigma)}
    \end{cases}
\end{equation*}
for all $N\geq N(j)$. With this choice, we then have
\begin{equation*}
    \sup_{x\in K}
    \abs{\langle \ell_j,\sigma(A_jx+b_j)\rangle-\langle\ell_j\circ\Pi_N,\sigma(
    \Pi_{N}A_j\Pi_{N}x+\Pi_{N}b_j)\rangle}
    <\delta/M. 
\end{equation*}
Therefore, setting $N_\ast:=\max\{N(1),\dots,N(M)\}$, we conclude that for all $N\geq N_\ast$
\begin{equation*}
    \sup_{x\in K}\abs{f(x)-\sum_{j=1}^M\langle \ell_j\circ\Pi_N,\sigma(
    \Pi_{N}A_j\Pi_{N}x+\Pi_{N}b_j)\rangle}<\varepsilon+\delta.
\end{equation*}

\end{proof}
Again, the required function $\mathcal{N}^{\varepsilon}: \X\rightarrow \R$ exists in view of Theorem \ref{prop: density}. However, we need to enhance Example \ref{ex: some concrete examples} to show that activation functions $\sigma$ satisfying condition \eqref{eq: Lipschitz 2} exist.
\begin{example}\label{ex: some concrete examples 2}
Let $\X$ be a real Fr\'echet space (not necessarily admitting a Schauder basis). Consider a function $\beta \in \operatorname{Lip} (\R;\R)$ such that 
\begin{equation*}
     \lim_{\xi\to\infty} \beta(\xi)=1,\;\lim_{\xi\to-\infty} \beta(\xi)=0, \;\beta(0)=0,
\end{equation*}
and arbitrary $z\in\X, z\neq 0$. Let $\psi\in\X'\setminus \{0\}$. Define
\begin{equation*}
    \sigma(x) = \beta(\psi(x))z,\quad x\in\X.
\end{equation*}
Evidently, $\sigma$ is continuous and von Neumann-bounded, because for any $j\in\N$
\begin{equation*}
    p_j(\sigma(x)) \leq \abs{\beta(\psi(x))} p_j(z)\leq \norm{\beta}_\infty p_j(z) <\infty
\end{equation*}
uniformly in $x\in\X$. Furthermore, it is clear that $\sigma$ satisfies \eqref{eq: abstract condition on sigma}. Let us finally check that condition \eqref{eq: Lipschitz 2} is met. To this aim, let $k\in\N$. We have
\begin{equation*}
    \begin{split}
        p_k(\sigma(x)-\sigma(y)) &= \abs{\beta(\psi(x))-\beta(\psi(y))} p_k(z)\\
        & \leq \operatorname{Lip}(\beta) p_k(z)) \abs{\psi(x)-\psi(y)}\\
        & \leq \operatorname{Lip}(\beta) p_k(z) C_{\psi}\, p_{m(\psi)}(x-y)\\
        & := C(\beta,z,\psi; k) \, p_{m(\psi)}(x-y), \quad x,y\in\X
    \end{split}
\end{equation*}
for some $m(\psi)\in\N$. Therefore, for any $k\geq m(\psi)$, since the seminorms are non-decreasing,
we have
\begin{equation*}
    p_k(\sigma(x)-\sigma(y)) \leq C(\beta,z,\psi; k) \, p_k(x-y), \quad x,y\in\X.
\end{equation*}

\end{example}

\section{Multi-layer Neural Networks}\label{sec:multilayer}
%\nils{Please check, I made some changes to this section (Results in structured in a proposition and removed the warmup case $n=2$). But the old part is still in the overleaf for the application paper and we can get it back.}
In this section we are going to show that results analogous to Theorems~\ref{prop: density} and ~\ref{prop: discriminatory condition is satisfied} hold also for multi-layer (deep) neural networks with a fixed number $n>1$ of layers. We consider the following $n$-layer neural network 
\begin{equation*}
    \mathcal{N}_{\ell,A_1,b_1,\dots ,A_n,b_n} : \X\to \mathbb F, \quad \mathcal{N}_{\ell,A_1,b_1,\dots ,A_n,b_n}(x):= \langle \ell,
    (\sigma\circ T_1 \circ \cdots \circ \sigma\circ T_n)(x)
    \rangle, \quad x\in\X,
\end{equation*}
with $\ell\in \X'$, $A_1,\dots ,A_n\in\mathcal{L}(\X)$, $b_1,\dots ,b_n\in\X$, $\sigma:\X\to\X$ continuous, and where we have set
\begin{equation*}
    T_j (x):= A_jx + b_j, \quad x\in\X,\, j=1,\dots,n.
\end{equation*}
Define 
\begin{equation*}
    \mathfrak N(\sigma) := \Span\{
    \mathcal{N}_{\ell,A_1,b_1,\dots ,A_n,b_n} ; \,\,\ell\in\X', A_1,\dots ,A_n\in\mathcal{L}(\X),b_1,\dots, b_n\in\X    \}.
\end{equation*}
Before embarking on the proof of the density of $\mathfrak N(\sigma)$, we need to establish the following result, which will turn out to be very fruitful in the sequel.

%\luca{We should change this condition and use the one proposed by Nils instead, namely the Separating property: indeed if, say, $u_+\notin \Span\{u_-,u_0\}$, then $u_+$ cannot be zero, and thus the proof is carried out in the same way as it is now. If $u_-\notin \Span\{u_+,u_0\}$, again same proof, but swapping the role of $\Psi_+$ and $\Psi_-$, i.e. we don't even need to treat this second case... Stimmt das? :)}\nils{Ich weiss es nicht! Do we not need below that $u_0=u_-=0$? Otherwise we arrive for $T_n(x) \in \Psi_-$ at $\langle \ell, (\sigma\circ T_1 \circ \cdots \circ \sigma)(T_{n-1}(u_-))
%     \rangle \, \mu[K\cap T_n^{-1}(\Psi_-)]$ and we need to choose $y_i$ iteratively also in this case but we might end up with an $y_1$ that is linearly dependent with the $y_1$ we got for $u_+$, and then we can not choose $l$ from Hahn-Banach so that it is zero when applied to say $u_-$ and $1$ when applied to $u_+$. Ok, maybe I need to think about it again tomorrow, I am a bit brain tired now. By the way, why do we need $\mathbb{F}=\mathbb{R}$ below?}
\begin{lemma}\label{lemma: sigma(Ay) not zero}
Assume that $\X$ is a real separable Fr\'echet space. Let $\sigma:\X\to\X$ be continuous and satisfying the following condition: there exist $\psi\in \X'\setminus\{0\}$ and $0\neq u_+\in\X$ such that 
\begin{equation*}
\begin{cases}
\lim_{\lambda\to\infty} \sigma(\lambda x) = u_+, \text{ if } x\in \Psi_+\\
\lim_{\lambda\to\infty} \sigma(\lambda x) = 0, \text{ if } x\in \Psi_-\\
\lim_{\lambda\to\infty} \sigma(\lambda x) = 0, \text{ if } x\in \Psi_0\\
\end{cases}
\end{equation*}
Let $0\neq y \in \X$ be arbitrary. Then there exists $A\in\mathcal{L}(\X)$ such that $\sigma(Ay)\neq 0$.
\end{lemma}
\begin{proof}
We need to distinguish two cases:
\begin{enumerate}
    \item $y \in \Psi_0$,
    \item $y \notin \Psi_0$.
\end{enumerate}
In the first case, let $\phi\in\X': \phi(y)\neq 0$. By Lemma \ref{lemma: trivial lemma}, choose $z$ accordingly, i.e. $\phi(z)=1,\psi(z)\neq 0$. Consider the projection onto $\Phi_0=\ker (\phi)$
\begin{equation*}
    \Pi_{\Phi_0}:\X\to \Phi_0, \quad x \mapsto x_{\Phi_0} = x - \phi(x)z,
\end{equation*}
which we know belongs to $\mathcal{L}(\X)$. Thus, $\psi(\Pi_{\Phi_0}y) = -\phi(y)\psi(z)\neq 0$, namely $\Pi_{\Phi_0}y \notin \Psi_0$. If $\Pi_{\Phi_0}y\in \Psi_+$, set $A=\lambda \Pi_{\Phi_0}$, where $\lambda>0$. Then $\sigma(\lambda \Pi_{\Phi_0}y)\to u_+\neq 0$ as $\lambda \to \infty$, and therefore for $\lambda \gg 0$ we obtain $\sigma(Ay)\neq 0$. If on the other hand $\Pi_{\Phi_0}y\in \Psi_-$, set $A=-\lambda \Pi_{\Phi_0}$ this time, to get the same conclusion, i.e. $\sigma(Ay)\neq 0$.

If $y\notin \Psi_0$, then define $A=\pm\lambda I$ with $\lambda\gg0$, accordingly if $y\in \Psi_+$ or $\in\Psi_-$.
\end{proof}
With this result at hand, we are now ready to prove:
\begin{proposition}\label{prop:density:multilayer}
Let $\X$ be a real and separable Fr\'echet space, and let $\sigma:\X\to\X$ be von Neumann-bounded and satisfy the conditions of Lemma~\ref{lemma: sigma(Ay) not zero}. Then $ \mathfrak N(\sigma)$ is dense in $C(\X;\mathbb R)$ with respect to the topology of compact subsets of $\X$.
\end{proposition}
\begin{proof}
Evidently, $\mathfrak N(\sigma)\subset C(\X;\mathbb R)$. Assume once again that $\text{cl}(\mathfrak N(\sigma))\subsetneq C(\X;\mathbb R)$. Then, once again we obtain the following 
\begin{equation*}
    \int_K \langle \ell,(\sigma\circ T_1 \circ \cdots \circ \sigma\circ T_n)(x)
    \rangle \mu(dx)=0,
\end{equation*}
for all $\ell\in \X'$, $A_1,\dots ,A_n\in\mathcal{L}(\X)$, $b_1,\dots ,b_n\in\X$. 

Observe that $\sigma(0)=0$. Reasoning as in the proof of Proposition \ref{prop: discriminatory condition is satisfied}, this time we get that, as $\lambda\to \infty$, pointwise in $x\in\X$,
\begin{equation*}
    \langle \ell,
    (\sigma\circ T_1 \circ \cdots \circ \sigma)(\lambda T_n(x))
    \rangle
    \to
     \begin{cases}
     \langle \ell, (\sigma\circ T_1 \circ \cdots \circ \sigma)(T_{n-1}(u_+))
     \rangle, \text{ if } T_n(x) \in \Psi_+\\
      \langle \ell, (\sigma\circ T_1 \circ \cdots \circ \sigma)(b_{n-1})
      \rangle, \text{ otherwise } \\
     \end{cases}
\end{equation*}
and hence, since $\sigma$ is von Neumann-bounded, by the dominated convergence theorem (for finite signed measures)
\begin{equation*}
\begin{split}
     &\langle \ell, (\sigma\circ T_1 \circ \cdots \circ \sigma)(T_{n-1}(u_+))
     \rangle \, \mu[K\cap T_n^{-1}(\Psi_+)] \\
      & \quad\quad+\langle \ell, (\sigma\circ T_1 \circ \cdots \circ \sigma)(b_{n-1})\rangle \left\{
\mu[K\cap T_n^{-1}(\Psi_-)] +
       \mu[K\cap T_n^{-1}(\Psi_0) ] \right\}= 0
\end{split}
\end{equation*}
for any $\ell\in \X'$, $A_1,\dots ,A_n\in\mathcal{L}(\X)$, $b_1,\dots ,b_n\in\X$. 

Choosing $b_1=b_2=\cdots =b_{n-1}=0$ results in 
\begin{equation*}
     \langle \ell, (\sigma\circ A_1 \circ \cdots \circ \sigma\circ A_{n-1})(u_+)
     \rangle \, \mu[K\cap T_n^{-1}(\Psi_+)]  = 0
\end{equation*}
for any $\ell\in \X'$, $A_1,\dots ,A_n\in\mathcal{L}(\X)$, and $b_n\in\X$. Define iteratively backward
\begin{equation*}
    \begin{cases}
    y_{n-1} = \sigma(A_{n-1}u_+)\\
    y_j = \sigma(A_jy_{j+1}), \quad j=1,\dots,n-2,
    \end{cases}
\end{equation*}
where $A_1,\dots,A_{n-1}\in\mathcal{L}(\X)$ are chosen in such a way that 
\begin{equation*}
    y_{n-1}\neq 0,\,y_{n-2}\neq 0,\dots ,\,y_1\neq 0.
\end{equation*}
This is achievable in virtue of Lemma \ref{lemma: sigma(Ay) not zero}. At the last step of the iteration we arrive at 
\begin{equation*}
     \langle \ell, y_1\rangle \, \mu[K\cap T_n^{-1}(\Psi_+)]  = 0
\end{equation*}
for any $\ell\in \X'$, $A_n\in\mathcal{L}(\X)$ and $b_n\in\X$, and hence $\mu[K\cap T_n^{-1}(\Psi_+)]=0$, namely 
\begin{equation*}
      \mu[K\cap A^{-1}(\Psi_+ +b)] =0
\end{equation*}
for any $A\in\mathcal{L}(\X)$, $b\in\X$. Following the steps in the proof of Proposition~\ref{prop: discriminatory condition is satisfied}, we conclude once more that $\mu=0$ and hence that $\mathfrak N(\sigma)$ is dense in $C(\X;\mathbb R)$. 
\end{proof}

As suggested by E \cite{WeinanE}, deep neural networks may be studied from the point of view of controlled ordinary differential equations (CODE). To recall, the $k$th layer input-output map can be represented as $x_{k+1}=x_k+V(x_k,\theta_k)$, to follow the notation in Cuchiero, Larsson and Teichmann \cite{Cuchiero2020DeepNN}. Here, $\theta_k$ is the affine map $T_k$ with training parameters $A_k$ and $b_k$, and $V$ is defined from the activation function $\sigma$,
$$
V(x_k,\theta_k):=\sigma(T_k x_k)-x_k.
$$
But then $x_k$ is the Euler discretisation scheme of the CODE
\begin{equation}
\label{CODE}
 \frac{dX_t}{dt}=V(X_t,\theta_t)\,,\quad t\in [0,1],
\end{equation}
which links the analysis of deep neural networks to CODE. We refer to Cuchiero, Larsson and Teichmann \cite{Cuchiero2020DeepNN} for a recent study for finite-dimensional deep neural networks. On the other hand, \eqref{CODE} provides a motivation for our definition of a (deep) neural network in infinite dimensions. Indeed, if we are in an infinite-dimensional vector space $\mathfrak{X}$ where we have available a theory for ODEs (a Banach space, say), then we read from \eqref{CODE} that $x\mapsto V(x,\theta_t)$ must map $\mathfrak{X}$ into itself. Thus, the activation function $\sigma$ is a mapping on $\mathfrak{X}$ into itself. The affine mapping $T$ (i.e., the control $\theta$) operates on $\mathfrak{X}$ as well. Thus, our proposed definition of (deep) neural networks in Fr\'echet spaces aligns naturally with CODEs.

\bibliography{literature}
\bibliographystyle{abbrv}
%%%%%%%%%%%%%%%%%%%%%%%%%%%%%%%%%%%%%%%%%%%%%
%%%%%%%%%%%%%%%%%%%%%%%%%%%%%%%%%%%%%%%%%%%%%
%%%%%%%%%%%%%%%%%%%%%%%%%%%%%%%%%%%%%%%%%%%%%

\end{document}